\documentclass[11pt,reqno]{amsart}

\usepackage{amsmath}
\usepackage{amsfonts}
\usepackage{amssymb}
\usepackage{amsthm}
\usepackage{graphicx}
\usepackage{epstopdf}
\usepackage{enumitem}
\usepackage{geometry}
\usepackage{hyperref}
\usepackage{tikz}
\usepackage[all]{xy}

\newtheorem{theorem}{Theorem}
\newtheorem{definition}[theorem]{Definition}

\newtheorem{example}[theorem]{Example}
\newtheorem{proposition}[theorem]{Proposition}
\newtheorem{remark}[theorem]{Remark}
\newtheorem{corollary}[theorem]{Corollary}
\newtheorem{lemma}[theorem]{Lemma}
\newtheorem{problem}[theorem]{Problem}

\begin{document}

\title{Topology on locally finite metric spaces}

\author{Valerio Capraro}
\address{University of Neuchatel, Switzerland}
\thanks{Supported by Swiss SNF Sinergia project CRSI22-130435}
\email{valerio.capraro@unine.ch}



\date{}

\maketitle

\begin{abstract}
The necessity of a theory of General Topology and, most of all, of Algebraic Topology on locally finite metric spaces comes from many areas of research in both Applied and Pure Mathematics: Molecular Biology, Mathematical Chemistry, Computer Science, Topological Graph Theory and Metric Geometry. In this paper we propose the basic notions of such a theory and some applications: we replace the classical notions of continuous function, homeomorphism and homotopic equivalence with the notions of NPP-function, NPP-local-isomorphism and NPP-homotopy (NPP stands for Nearest Point Preserving); we also introduce the notion of NPP-isomorphism. We construct three invariants under NPP-isomorphisms and, in particular, we define the fundamental group of a locally finite metric space. As first applications, we propose the following: motivated by the longstanding question whether there is a purely metric condition which extends the notion of amenability of a group to any metric space, we propose the property SN (Small Neighborhood); motivated by some applicative problems in Computer Science, we prove the analog of the Jordan curve theorem in $\mathbb Z^2$; motivated by a question asked during a lecture at Lausanne, we extend to any locally finite metric space a recent inequality of P.N.Jolissaint and Valette regarding the $\ell_p$-distortion.
\end{abstract}

\section{Introduction}

\subsection{Motivations, applications and related literature}

The necessity of a theory of General Topology and, most of all, of Algebraic Topology on locally finite metric spaces comes from several areas of research:

\begin{itemize}
\item Molecular biologists, as Bon, Vernizzi, Orland and Zee, started in 2008 a classification of RNA structures looking at their \emph{genus} (see \cite{Bo-Ve-Ze08} and also \cite{Bo-Or11} and \cite{RHAPSN11} for developments).
\item Chemists, as Herges, started in 2006 to study molecules looking like a \emph{discretization} of the M\"{o}bius strip or of the Klein bottle (see \cite{He06} and also the book \cite{Ro-Ki02}).
\item Computer scientists founded in the late 80s a new field or research, called Digital Topology, motivated by the study of the topology of the screen of a computer, which is a locally finite spaces whose points are pixels (see, for instance, \cite{Ko}).
\end{itemize}

Although the motivations are mainly from Applied Mathematics, the theory has many applications also in Pure Mathematics.

\begin{enumerate}
\item The theory, introducing a notion of connectedness for locally finite metric spaces, gives the tools to \emph{bring up} many results from Graph Theory to locally finite metric spaces. In particular, motivated by a question asked during a lecture of Alain Valette, we prove a version for locally finite metric spaces of a recent inequality proved by P.N.Jolissaint and A.Valette regarding the $\ell^p$-distortion.
\item The theory, introducing a notion of continuity for locally finite metric spaces, gives the tools to \emph{bring down} many results from, say, Topology of Manifolds to locally finite metric spaces. In particular, motivated by a problem in Computer Science, we prove the analogue of the Jordan curve theorem in $\mathbb Z^2$.
\item The theory leads to the discovery of the property SN, which is, as far as we know, the first known \emph{purely metric} property that reduces to amenability for the Cayley graph of a finitely generated group.
\end{enumerate}

Here is a little survey about related ideas in literature.

\begin{itemize}
\item Digital Topology contains many good ideas that have led to the discovery of Algebraic Topology on the so-called digital spaces (see \cite{Kh87}, \cite{Ko89}, \cite{ADFQ00} and \cite{Ha10}). Unfortunately, this theory, despite concerns only subsets of $\mathbb Z^n$ (see \cite{ADFQ00}, Example 2.1), is quite technical, far from being intuitive and contains weird behaviors, as the fact that the digital fundamental group does not always verify the multiplicative property (see \cite{Ha10} and reference therein) and the Seifert-van Kampen theorem (see again \cite{ADFQ00}).
\item Topological Graph Theory is based on the following identification: every vertex is seen as a point; any edge joining two vertex is seen as a copy of $[0,1]$ joining the two extremal points. This theory has many applications that one can find in any of the books on the topic, but in our opinion it does not capture the real topological essence of a graph. We will see that there are technical and \emph{natural} reasons to expect that 3-cycle and the 4-cycle must have trivial fundamental group. The latter example also shows that Rips complexes are not enough to reach our purpose\footnote{Trying to define the fundamental group via Rips complexes leads to many difficulties: first of all, since it depends on the radius $\delta$, it would not be clear which is the \emph{right} fundamental group; moreover, a non-trivial definition would imply that the fundamental group of the unit square $Q$ in $\mathbb Z^2$ with the Euclidean metric is $\mathbb Z$, contradicting, first of all, the natural requirement of the validity of the Jordan curve theorem in $\mathbb Z^2$ (indeed $\mathbb Z^2\setminus Q$ would be \emph{path-connected}), but also the intuition that an observer living on $Q$ has no unit of measurement to find a hole.}.
\item Many authors have studied coarse cohomology on locally finite metric spaces (see \cite{Ho72}, \cite{Ro93} and, for a systematic approach, \cite{Ro03}). This is of course an interesting theory, with many applications in K-theory, where one is interested only on large-scale behaviors of a space, but is useless in our case where one is interested in local behaviors and, more specifically, in finite spaces (molecules are finite, the screen of a computer is finite ...), that get trivial in coarse geometry.
\end{itemize}

Hoping to have convinced the reader that something new is really necessary, the very basic idea that we are going to follow is a change of perspective about the notion of continuity. Let $X$ be a topological space, a continuous path in $X$ is classically seen as a continuous function $f:[0,1]\rightarrow X$. Hence, the notion of continuity is, in some sense, imposed from \emph{outside}, taking as a unit of reference the interval $[0,1]$. This choice works very well for manifolds, because every open set contains copies of $[0,1]$, but it does not work for locally finite metric spaces. The idea is to change perspective looking at the notion of continuity from \emph{inside}. Continuity is \emph{to do steps as short as possible}: a continuous path in a locally finite metric space will be constructed, roughly speaking, as follows: start from a point $x_0$, move to one of the nearest points $x_1$ and so on. From a philosophical point of view, this paper is a corollary of this simple change of perspective.

\subsection{Structure of the paper}

\begin{itemize}
\item In Sec. \ref{se:algebraictopology} we develop the basics of Algebraic Topology on locally finite metric spaces: we introduce the notion of fundamental group of a locally finite metric space, taking inspiration from the original definition by Poincar\`{e}, and we prove that the construction is independent on the base point and gives an NPP-invariant (see Theorem \ref{th:fundamentalgroup}). We give some examples of fundamental groups, showing that this fundamental group behaves like a \emph{discretization} of the classical fundamental group: for instance, the fundamental group of the grid $\mathbb Z^2$ is the same as the classical fundamental group of $\mathbb R^2$, the fundamental group of $\mathbb Z^2\setminus\{(0,0)\}$ is the same as the classical fundamental group of $\mathbb R^2\setminus\{(0,0)\}$. We state some results about the fundamental group that will be proved in the companion piece of this paper \cite{Ca-Go-Pi11}. We introduce the discrete analog of the notions of continuous function, homeomorphism and homotopic equivalence and we prove, in Theorem \ref{th:homotopicequivalence}, that they induce group homomorphisms/isomorphisms at the level of the fundamental groups.
\item In Sec. \ref{se:firstapplications} we propose two applications of the fundamental group: motivated by a problem in Computer Science, we  prove the analog in $\mathbb Z^2$ of the Jordan curve theorem (see Theorem \ref{th:jordanbrouwer}); motivated by a question asked during a lecture we extend to any finite metric space a recent result by P.N. Jolissaint and A. Valette, which gives a lower bound for the $\ell^p$-distortion of a finite connected graph (see Theorem \ref{th:jolivale} and Corollary \ref{cor:jolivale}).
\item Motivated by the theory developed in Sec. \ref{se:algebraictopology}, in Sec. \ref{se:definitions} we introduce a new notion of isomorphism between locally finite metric spaces, which will be called NPP-isomorphism. In Sec. \ref{se:examples} we collect some simple examples of NPP-isomorphisms, showing also that this NPP-way to embed a metric space into another is not equivalent to any of the other well-known ways. In particular, we introduce the notion of \emph{graph-type} metric space and we prove that any such a metric space is canonically NPP-isomorphic to a \emph{Traveling Salesman} graph (Theorem \ref{th:NPPisomorphism1}).
\item In Sec. \ref{se:isoperimetric} we introduce the second NPP-invariant for locally finite metric spaces (the fundamental group being the first one), which we call \emph{isoperimetric constant} (Definitions \ref{def:isoperimetricspace}, \ref{def:globaliso} and Theorem \ref{th:isoinvariant}), since the construction is similar (and reduces) to the isoperimetric constant of a locally finite, infinite, connected graph (that is not an NPP-invariant, by the way).
\item Motivated by the question \emph{What is the property that corresponds to have isoperimetric constant equal to zero}, in Sec. \ref{se:sn} we introduce a property for any metric space, the Small Neighborhood property (property SN), in Definition \ref{def:small}. At first we give some examples of spaces with or without property SN (Proposition \ref{prop:example}) and then we study the relation between the property SN and the amenability: we prove that the property SN is equivalent to the amenability for infinite, locally finite, connected graph (Theorem \ref{th:snvsamenability} and Proposition \ref{prop:caprarovsceccherini}). So, property SN is the first known purely metric property that reduces to amenability for (the Cayley graph of) a finitely generated group. Finally, we give a description of the property SN of a locally finite space in terms of the isoperimetric constant (Theorem \ref{th:isoperimetricvssn}), which is a version for locally finite metric spaces of well-known results by Ceccherini-Silberstein, Grigorchuk, de la Harpe and Elek, S\'{o}s. In particular, it turns out that the property SN is an NPP-invariant.
\item In Sec. \ref{se:zoomiso} we introduce the third invariant of the theory, which we call \emph{zoom isoperimetric constant}, since it gives more precise information than the isoperimetric constant, using a construction that, in some sense, increases the knowledge of the space, step by step (Definition \ref{def:zoom} and Theorem \ref{th:zoominvariant}). We prove that the two isoperimetric constants are, in some sense, independent, meaning that in general there is no way to deduce one from the other (Theorem \ref{th:invariantindependent}). Using this new invariant, we propose a stronger notion than amenability, the notion of \emph{locally amenable space}, a space which is amenable with respect to any \emph{local observation}. We prove in Theorem \ref{th:localglobalamenability} that in case of graph-type spaces, a single observation is enough to conclude if the space is completely amenable. This suggests (Sec. \ref{suse:expanders}) some way to extend the classical notion of expander of graphs and the recent notion of geometric property (T) to any locally finite metric space.
\end{itemize}

Before starting the technical discussion, let us fix some notation. Throughout the paper:
\begin{itemize}
\item Given a finite set $A$, we denote by $|A|$ the number of elements in $A$.
\item $(X,d)$ will be, a priori, any metric space. Given $x\in X$ and $r>0$, $B(x,r)$ will denote the open ball of radius $r$ about $x$.
\item Given a subset $A$ of $X$, we denote by $\overline A$ the closure of $A$.
\item Given $A\subseteq X$ and $\alpha>0$ we denote
\begin{itemize}
\item $cN_\alpha(A)=\{x\in X:d(x,A)\leq\alpha\}$
\item $cN_\alpha(A)^\circ=\{x\in X:d(x,A)<\alpha\}$
\end{itemize}
The difference $cN_\alpha(A)\setminus\overline{cN_\alpha(A)^\circ}$ is denoted by $cB_\alpha(A)$.
\end{itemize}

\subsection{Other possible applications}\label{suse:otherapplications}

This subsection can be skipped at first reading. It is little speculative and shows the picture that the author had in mind writing this paper. The (hopefully) interesting part is the formulation of a new class of problems that might be of interest both for Pure and Applied Mathematics.
\\\\
Let $(X,d)$ be a metric space and suppose that there is an observer that wants to make a map of $X$. The strong restriction, that often an observer faces in the real life, is that the distance between two points is not known exactly. The observer is only able to say, given two points $y,z$, which one is nearer to another point $x$. One can construct dozens of real examples where such a situation may happen; here are a couple of them:

\begin{itemize}
\item An astronomer wants to make a picture of the universe. He or she does not know the exact distance between two objects, but (sometimes) he can decide indirectly which one is nearer to a third object.
\item An applied computer scientist wants to upload some data related to a metric space. Maybe in this case he knows mathematically the function $d$, but it could be impossible to upload it exactly (maybe just because $d(x,y)$ is an irrational number and (s)he is allowed to use just rational numbers with a fixed number of digits). His/her idea might be to replace $d$ with another function which preserves the relation of nearness/farness between points.
\item A traveling salesman who does not know the exact distance between the cities that he has to visit, but knows, maybe by mere observation of a map, their relations of nearness/farness.
\end{itemize}

We will formalize later the intuitive idea that, abstractly, the procedure of making such a map is described by an NPP-embedding, namely an embedding of $(X,d)$ into another metric space $(X',d')$ which is \emph{Nearest Points Preserving}. The image of $(X,d)$ under this embedding is called \textbf{symbolic map}.
The theory has clearly two general questions:

\begin{itemize}
\item Is it possible (and, in case, how?) to make a \emph{good} symbolic map?
\item What kind of information are encoded in a symbolic map?
\end{itemize}

These questions generate a new class of mathematical questions:

\begin{itemize}
\item How to find good NPP-embeddings?
\item What are the invariants under NPP-isomorphisms?
\end{itemize}

The second question will be widely discussed in this paper, whose primary purpose is to find three different invariants under NPP-isomorphisms. The first question will not be discussed and so we list below some explicit questions related to it.

\begin{problem}\label{pr:embedding}
Given a metric space, find its best NPP-embedding (Banach space? lowest dimensional Hilbert space? constant curvature surfaces?
\end{problem}

As already happened a couple of centuries ago, also in this case Cartography helps to formulate interesting and difficult questions: the general problem in cartography is to represent some metric space into a map that we can easily construct and bring with us. It is clear that the best map would be a planar map, but we can be satisfied also of a surface of constant curvature. So an interesting problem for the application would be

\begin{problem}
Find hypothesis that guarantee the existence of an NPP-embedding into a surface of constant curvature and, in particular, into the Euclidean plane.
\end{problem}

A classical problem in Applied Computer Science is that of uploading some data in such a way to keep the relations of farness and nearness among points. Since in this case we can only use a finite number of digits, say $N$, the most important problem seems to be the following:

\begin{problem}
Find hypotheses that guarantee the existence of an NPP-embedding of a (necessarily finite) metric space into some $(\mathbb (\mathbb Q_N)^n,d)$, where $\mathbb Q_N$ is the set of rational numbers with $N$ digits and $d$ is a computable metric (for instance, the $\ell_1$-metric).
\end{problem}

Another interesting problem regards the possibility to find the best NPP-embedding

\begin{problem}
Find hypotheses that guarantee the existence of an NPP-embedding into some $\mathbb Z^n$, equipped with some \emph{computable metrics}, as for instance the $\ell^1$-metric.
\end{problem}

In this latter case, it is likely that a theory of homology of a locally finite metric space can help.

\section{Basics of Algebraic Topology on locally finite spaces}\label{se:algebraictopology}

This section is devoted to the basic concepts of General Topology and Algebraic Topology on a locally finite metric space: we introduce the notion of fundamental group, we present some basic properties and we introduce the discrete analog of the classical notions of continuous function, homeomorphism and homotopy.

\subsection{The fundamental group of a locally finite space}\label{se:fundamental}

In the spirit of the original definition of the fundamental group of a topological space, we construct a group whose role is to capture the holes of a locally finite metric space. The basic example to keep in mind is the following: $\mathbb Z^2$ has no holes, $\mathbb Z^2\setminus\{(0,0)\}$ has a hole.
\\\\
Try to attach a fundamental group to a locally finite metric space is not a new idea, at least when the metric comes from a locally finite, connected graph. In this case the standard definition, using spanning trees, is not that interesting. Very recently, Diestel and Spr\"{u}ssel have studied a more interesting notion of fundamental group for locally finite connected graphs with ends (see \cite{Di-Sp11}). On the other hand, many authors have studied some coarse cohomogy (see \cite{Ho72}, \cite{Ro93} and, for a systematic approach, Chapter 5 in \cite{Ro03}). As we told in the introduction, our motivating problems force to consider something new, something able to capture the details of the space. We recall that our main intuition is to change perspective looking at the notion of continuity from \emph{inside}. So, we replace the usual notion of a continuous path with the notion of a path made by the shortest admissible steps and the notion of homotopic equivalence with an equivalence that puts in relation paths which differ by the shortest admissible steps.
\\\\
Throughout this section, let $(X,d)$ be a locally finite metric space; i.e. a metric space whose bounded sets are finite.

\begin{definition}\label{def:1discrete}
Let $x\in X$. The set $dN_1(x)$, called \textbf{discrete 1-neighborhood of $x$}, is constructed as follows: let $r>0$ be the smallest radius such that $|B(x,r)|\geq2$, then $dN_1(x)=B(x,r)$.
\end{definition}

\begin{definition}
Let $(X,d)$ be a locally finite metric spaces. A \textbf{continuous path} in $X$ is a sequence of points $x_0x_1\ldots x_{n-1}x_n$ such that for all $i=1,2,\ldots,n$, one has
\begin{equation}\label{eq:continuouspath}
x_i\in dN_1(x_{i-1}) \text{ and } x_{i-1}\in dN_1(x_i)
\end{equation}
\end{definition}

Observe that \emph{to be joined by a continuous path} is an equivalence relation. The equivalence classes are called \textbf{path-connected components}.

\begin{remark}\label{ex:pathconnected}
A simple but important remark is that a continuous path $x_0x_1\ldots x_{n-1}x_n$ is characterized by the property: $d(x_i,x_{i-1})=d(x_{i+1},x_i)$, for all $i\in\{1,\ldots,n-1\}$. So, the following set of points inside $\mathbb R^2$ with the Euclidean distance is not path connected

$$
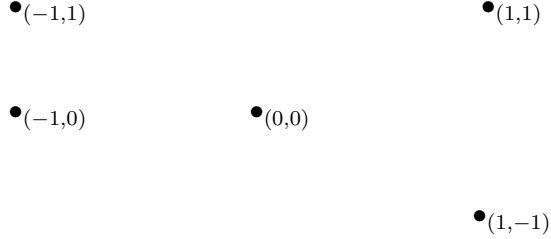
\xymatrix{
        \bullet_{(-1,1)}  && &&\bullet_{(1,1)}\\
        \bullet_{(-1,0)}  && \bullet_{(0,0)} &&\\
        && && \bullet_{(1,-1)}}
$$

and, in particular, it has three path connected components: $C_1=\{(-1,1),(-1,0),(0,0)\}$, $C_2=\{(1,1)\}$ and $C_3=\{(1,-1)\}$. An example of path-connected subset $X$ of $\mathbb R^2$ (with the Euclidean distance) is the following:

$$
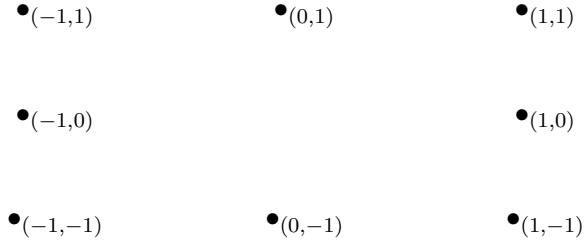
\xymatrix{
        \bullet_{(-1,1)}  && \bullet_{(0,1)} && \bullet_{(1,1)}\\
        \bullet_{(-1,0)}  &&  && \bullet_{(1,0)}\\
        \bullet_{(-1,-1)} && \bullet_{(0,-1)} && \bullet_{(1,-1)}}
$$
This latter example shows also that a locally finite path-connected metric space might not look like a graph: there is no way to put on $X$ a graph distance which is coherent to the one inherited by the Euclidean metric.

\end{remark}

As for the classical fundamental group, let us restrict our attention to path-connected locally finite metric spaces\footnote{It is a trivial fact, but important to keep in mind, that a path-connected metric space can be very far from being a graph. Consider for instance $\mathbb Z^2\setminus\{(0,0)\}$, equipped with the Euclidean metric.} $X$ and let $x\in X$ be a fixed base point. A \textbf{continuous circuit} is simply given by a continuous path $x_0x_1\ldots x_{n-1}x_n$ such that $x_0=x_n=x$. As usual, we equip the set of circuits with the same base point with the operation of concatenation
$$
(x_0x_1\ldots x_{n-1}x_n)(y_0y_1\ldots y_{m-1}y_m)=x_0\ldots x_ny_0\ldots y_m
$$
with the computational rule

\begin{equation}\label{eq:rule}
x_ix_i=x_i \text{ for all } x_i\in X
\end{equation}

Now let us introduce a notion of homotopic equivalence. Consider two  continuous circuits $x_0x_1\ldots x_{n-1}x_n$ and $y_0y_1\ldots y_{m-1}y_m$. By adding some $x$'s or some $y$'s we can suppose that $n=m$ (thanks to the rule \ref{eq:rule}).

\begin{definition}
The two circuits $x_0x_1\ldots x_{n-1}x_n$ and $y_0y_1\ldots y_{n-1}y_n$ are said to be homotopic equivalent if for all $i=1,\ldots,n-1$, there is a finite sequence $z_i^1,\ldots,z_i^k$ ($k$ does not depend on $i$, \emph{a posteriori} thanks to the computational rule in \ref{eq:rule}) such that

\begin{enumerate}
\item $z_i^1=x_i$
\item $z_i^k=y_i$
\item $z_i^{h+1}\in dN_1(z_i^h)$ and $z_i^h\in dN_1(z_i^{h+1})$, for all $h=1,\ldots,k-1$
\item $x_0z_1^hz_2^h\ldots z_{n-1}^hx_n$ is a continuous circuit for all $h=1,\ldots k$
\end{enumerate}
\end{definition}

A comfortable way to look (and also to define) an homotopic equivalence is by building the \textbf{homotopy matrix}

$$
\left(
  \begin{array}{ccccc}
    x_0 & x_1 & \ldots & x_{n-1} & x_n \\
    x_0 & z_1^2 & \ldots & z_{n-1}^2 & x_n \\
    \ldots & \ldots & \ldots & \ldots & \ldots \\
    x_0 & z_1^{k-1} & \ldots & z_{n-1}^{k-1} & x_n \\
    x_0 & y_1 & \ldots & y_{n-1} & x_0 \\
  \end{array}
\right)
$$
with the condition that every row and every column form a continuous path in $X$.

\begin{example}\label{ex:homotopy}
Let $Q=\{(0,0),(1,0),(1,1),(0,1)\}$ as a subset of $\mathbb R^2$ with the Euclidean metric and consider the continuous circuit $\gamma$ with base point $(0,0)$ as in the figure:
$$
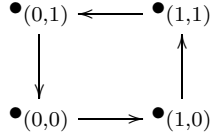
\xymatrix{
        \bullet_{(0,1)}\ar@{->}[d]  & \bullet_{(1,1)}\ar@{->}[l]\\
        \bullet_{(0,0)}\ar@{->}[r]  & \bullet_{(1,0)}\ar@{->}[u]}
$$
Now, the definition, says that it is possible replace the points $(0,1)$ and $(1,1)$ with one of closest points, for instance $(0,0)$ and $(1,0)$. Hence, $\gamma$ is homotopic equivalent to the following circuit

$$
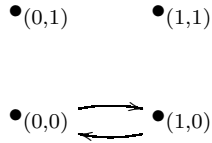
\xymatrix{
        \bullet_{(0,1)} & \bullet_{(1,1)}\\
        \bullet_{(0,0)}\ar@/^/[r] & \bullet_{(1,0)}\ar@/^/[l]}
$$

Now, we can replace $(1,0)$ with one of the nearest point, for instance $(0,0)$. In this way, we obtain the constant path in $(0,0)$. So, the initial circuit is homotopic equivalent to the constant path in $(0,0)$. In terms of homotopic matrix we have

$$
\left(
  \begin{array}{ccccc}
    (0,0) & (1,0) & (1,1) & (0,1) & (0,0) \\
    (0,0) & (1,0) & (1,0) & (0,0) & (0,0) \\
    (0,0) & (0,0) & (0,0) & (0,0) & (0,0) \\
  \end{array}
\right)
$$
This is a \emph{rough proof} that $\pi_1(C_4)=\{0\}$, where $C_4$ is the 4-cycle. Notice that such a procedure cannot be adapted for $C_n$, with $n\geq5$.
\end{example}

It should be clear that everything goes well. Let us summarize all basics properties with the following

\begin{theorem}\label{th:fundamentalgroup}
\begin{enumerate}
\item The homotopic equivalence is an equivalence relation on the set of circuits with base point $x$. The quotient set is denoted by $\pi_1(X,x)$.
\item The operation of concatenation is well defined on $\pi_1(X,x)$, giving rise to the structure of group.
\item $\pi_1(X,x)$ does not depend on $x$ in the following sense: given another $y\in X$, there is a group isomorphism between $\pi_1(X,x)$ and $\pi_1(X,y)$.
\end{enumerate}
\end{theorem}

\begin{proof}
(1) and (2) are evident. (3) is the \emph{discretization} of the standard proof: let $c_0c_1\ldots c_{n-1}c_n$ be a continuous path connecting $x$ with $y$, then the mapping $\Phi:\pi_1(X,x)\rightarrow\pi_1(X,y)$ defined by
$$
\Phi([x_0x_1\ldots x_{n-1}\_n])=[c_nc_{n-1}\ldots c_1c_0x_0x_1\ldots x_{n-1}x_nc_0c_1\ldots c_{n-1}c_n]
$$
is the desired isomorphism.
\end{proof}

Let us give some basic examples, just to show how the fundamental group behaves:

\begin{example}\label{ex:fundamentalgroups}
\begin{itemize}
\item The fundamental group of $\mathbb Z^2$ is trivial and the fundamental group of $\mathbb Z^2\setminus\{(0,0)\}$ is equal to $\mathbb Z$ (both with respect to the graph metric and the Euclidean one). These two examples also show that this way to do algebraic topology on locally finite spaces is not coarsely invariant. This is natural since the details of the space are very important for our motivating problems.
\item Let $C_n$ be the $n$-cycle graph, then
$$
\pi_1(C_n)=\left\{
             \begin{array}{ll}
               \{e\}, & \hbox{if $n\leq4$;} \\
               \mathbb Z, & \hbox{if $n\geq5$.}
             \end{array}
           \right.
$$
The proof is basically the generalization of the argument in Example \ref{ex:homotopy}. In some sense, three of four points are not enough to \emph{find a hole} and one needs more. This is clarified in much more generality in the following Theorem \ref{th:discretization}.
\end{itemize}
\end{example}

We conclude this introductory section stating some results that will be proved in the companion piece of this paper \cite{Ca-Go-Pi11}.

\begin{theorem}\label{th:discretization}
Let $X$ be a compact metrizable path-connected manifold and let $\mathcal T$ be a fine enough triangulation of $X$. Equip the 1-skeleton of $\mathcal T$ with the natural structure of finite connected graph. Then
$$
\pi_1(\mathcal T)=\pi_1(X)
$$
\end{theorem}

The first application of this result is to make us optimistic about the theory: it is not \emph{orthogonal} to the classical one, but we obtain it as a \emph{limit theory}. Another application is that it allows to construct (even explicitly) finite graphs whose fundamental group has torsion, which is a completely new phenomenon in Topological Graph Theory. Finally, this result gets very close to the construction of an algorithm for \emph{computing} the fundamental group of a manifold, knowing a triangulation, which is an open problem for $n\geq4$ (see \cite{KJZLG08}). It is clear that such an algorithm will not work in general, because the word problem on finitely presented groups is not solvable by a celebrated result by Novikov (see \cite{No55}), but this is of course a interesting topic for further research.

\begin{theorem}
Let $X,Y$ be two locally finite path-connected metric space in normal form\footnote{Let $X$ be a locally finite path-connected metric space and $x\in X$. Let $R_x$ be the smallest $R>0$ such that $|B(x,R)|\geq2$. One can easily prove that $R_x$ does not depend on $x$ and so it is a constant of the space, called \textbf{step of $X$}. A locally finite path-connected space is said to be in \textbf{normal form} if the distance is normalized in such a way that the step is equal to $1$. Observe that one can pass from a general form to the normal form via a simple re-normalization, which does not affect the fundamental group (see Theorem \ref{th:homotopicequivalence}). It follows that the hypothesis in the Theorem is \emph{without loss of generality}.}. Then $X\times Y$ is path-connected with the $\ell^1$-metric. One has
$$
\pi_1(X\times Y)=\pi_1(X)\times\pi_1(Y)
$$
\end{theorem}

\begin{theorem}
Let $X$ be a locally finite path-connected metric space and $\{U_\lambda\}$ be a continuous covering\footnote{The notion of continuous covering is the discrete counterpart of an open covering, closed under finite intersections and made of path-connected subsets, which is the standard hypothesis of Seifert-van Kampen's theorem.} of $X$. Then $\pi_1(X)$ is uniquely determined by the $\pi_1(U_\lambda)$'s in Seifert-van Kampen's sense.
\end{theorem}

\subsection{Continuous functions, homeomorphisms and homotopic equivalences between locally finite metric spaces}\label{suse::discretecontinuity}

In this section we want to introduce the discrete analog of the classical notions of continuous function, homeomorphism and homotopic equivalence.
\\\\
In order to find the discrete analog of the classical notion of continuous function, it suffices to \emph{discretize} the standard notion of continuous function from a metric space to another. We get

\begin{definition}\label{def:nppfunction}
A function $f:X\rightarrow Y$ between two locally finite metric spaces is called \textbf{NPP-function} if for all $x\in X$ one has
$$
f(dN_1(x))\subseteq dN_1(f(x))
$$
\end{definition}

Observe that NPP-functions map continuous circuits to continuous circuits (exactly as continuous functions do in Topology of Manifolds). Hence, the image of a locally finite path-connected metric space is still so. In particular, if $X$ is a locally finite path-connected metric space and $f:X\rightarrow\mathbb Z$ is an NPP-function. Suppose that there are $x_0,x_1\in X$ such that $f(x_0)f(x_1)<0$, then there is $x\in X$ such that $f(x)=0$. This is a first very simple example showing how the discrete analog of classical results for continuous functions hold for NPP-function. By the way, notice that such a result is not true for continuous function in classical sense (just think of $f:\mathbb Z\rightarrow\mathbb Z$ defined by $f(x)=x^2-2$). Another simple remark about NPP-function is that composition of NPP-function is still an NPP-function. This simple property will be often used in the sequel without further comment.

\begin{definition}\label{def:homeo}
An NPP-function $f$ between two locally finite metric spaces $X$ and $Y$ is called \textbf{NPP-local-isomorphism} if it is bijective and the inverse $f^{-1}$ is still an NPP-function.
\end{definition}

The adjective \emph{local} will be clarified in Sec. \ref{se:definitions}. At the moment, we can just say that it comes from the fact that such functions look just at the smallest non-trivial neighborhood of a point $x$. In order to study some properties we will need some more global notion. Let us observe explicitly that the requirement that $f^{-1}$ is NPP is necessary. Indeed, consider $X=\{0,1,2\}$ with the Euclidean metric and let $\Delta$ be the regular triangle in $\mathbb R^2$. Any bijection $f:X\rightarrow\Delta$ is an NPP-function, but none of them is an NPP-local-isomorphism.

\begin{definition}\label{def:homotopyequivalence}
Let $f,g:X\rightarrow Y$ be two NPP-functions. They are said to be \textbf{homotopic equivalent} if there is a sequence of NPP-functions $f_n:X\rightarrow Y$ such that
\begin{enumerate}
\item $f_1(x)=f(x)$, for all $x\in X$
\item for all $x\in X$, there is a positive integer $k(x)$ such that $f_1(x)f_2(x)\ldots f_{k(x)}(x)$ is a continuous path joining $f(x)$ and $g(x)$.
\end{enumerate}
If $f$ is homotopic to $g$ we write $f\simeq g$.
\end{definition}

\begin{remark}
Notice that every point $f(x)$ is transported continuously in $g(x)$ in finite time, but the sequence $f_n$ might be infinite. As an explicit example, we now show that the identity on $\mathbb Z$ is homotopic equivalent to the constant path at the origin (hence $\mathbb Z$ is contractible). Define $f_n:\mathbb Z\rightarrow\mathbb Z$ to be
$$
f_n(x)=\left\{
         \begin{array}{ll}
           0, & \hbox{if $|x|\leq n$} \\
           x+n, & \hbox{if $x<-n$}\\
           x-n, & \hbox{if $x>n$}
         \end{array}
       \right.
$$
It is clear that each $f_n$ is an NPP-function and that the sequence $f_0(n)f_1(n)\ldots f_n(n)$ is a continuous path connecting $0$ with $n$, as claimed.
\end{remark}

\begin{definition}\label{def:homotopytype}
Two locally finite metric spaces are called \textbf{homotopic equivalent} if there are NPP-functions $f:X\rightarrow Y$ and $g:Y\rightarrow X$ such that $f\circ g\simeq Id_Y$ and $g\circ f\simeq Id_X$.
\end{definition}

Now we want to prove the discrete analog of well-known properties of Algebraic Topology; namely, that continuous functions induce group homomorphisms at the level of the fundamentals group and that homeomorphisms and homotopic equivalences induce isomorphisms at the level of the fundamental groups. The last result is a first example of \emph{how to discretize an existing proof}. In particular, we present the discretization of the proof in \cite{Ma}, Chapter II, Section 8.
\\\\
Let us fix some notation: let $X,Y$ be two locally finite path-connected spaces and $f,g:X\rightarrow Y$ two homotopic equivalent NPP-functions. Let $\{f_n\}$ be the sequence of NPP-functions describing the homotopic equivalence between $f$ and $g$. Fix a base point $x_0\in X$ and let $\gamma=f_1(x_0)f_2(x_0)\ldots f_{n-1}(x_0)f_n(x_0)$ be the continuous path connecting $f(x_0)$ and $g(x_0)$ along the homotopic equivalence. Define the isomorphism (as in Theorem \ref{th:fundamentalgroup}) $u:\pi_1(Y,f(x_0))\rightarrow\pi_1(Y,g(x_0))$ to be $u(\beta)=\gamma^{-1}\beta\gamma$.

\begin{lemma}\label{lem:commutativediagram}
The following diagram is commutative

$$
\xymatrix{\pi_1(X,x_0)\ar[r]^{f_*}\ar[dr]_{g_*} & \pi_1(Y,f(x_0))\ar[d]^u\\
& \pi_1(Y,g(x_0))}
$$

\end{lemma}

\begin{proof}
Let $\alpha\in\pi_1(X,x_0)$ and let us prove that $g_*(\alpha)=\gamma^{-1}f_*(\alpha)\gamma$. Let $\alpha$ be represented by the continuous circuit $x_0x_1\ldots x_{n-1}x_n$. Adding,if necessary, some constant paths, we can build the following homotopic matrix
$$
\left(
  \begin{array}{ccccc}
    f(x_0) & f(x_1) & \ldots & f(x_{n-1}) & f(x_n) \\
    f_2(x_0) & f_2(x_1) & \ldots & f_2(x_{n-1}) & f_2(x_n) \\
    \ldots & \ldots & \ldots & \ldots & \ldots \\
    f_{n-1}(x_0) & f_{n-1}(x_1) & \ldots & f_{n-1}(x_{n-1}) & f_{n-1}(x_n) \\
    g(x_0) & g(x_1) & \ldots & g(x_{n-1}) & g(x_n) \\
  \end{array}
\right)
$$
where the terminology \emph{homotopic matrix} means that every row and every column is a continuous path. Observe that if we read around the boundary of this matrix we have exactly $f_*(\alpha)\gamma (g_*(\alpha))^{-1}\gamma^1$. The fact that the matrix is an homotopic matrix implies that this circuits is homotopic equivalent to the constant path (just eliminate each column starting from the right side and, finally, collapse the first column to the constant path). Hence $f_*(\alpha)\gamma (g_*(\alpha))^{-1}\gamma^1=1$, as required.
\end{proof}

\begin{theorem}\label{th:homotopicequivalence}
Let $f:X\rightarrow Y$ be an NPP-function between two locally finite path-connected metric spaces. Consider the induced function $f_*:\pi_1(X,x)\rightarrow\pi_1(Y,f(x))$ defined by
$$
f_*[x_0x_1\ldots x_{n-1}x_n]=[f(x_0)f(x_1)\ldots f(x_{n-1})f(x_n)]
$$
\begin{enumerate}
\item $f_*$ is always well-defined and it is always a group homomorphism.
\item If $f$ is an NPP-local-isomorphism, then $f_*$ is a group isomorphism.
\item If $f$ is an homotopic equivalence, then $f_*$ is a group isomorphism.
\end{enumerate}
\end{theorem}

\begin{proof}
\begin{enumerate}
\item $f_*$ is well-defined, since it maps circuits to circuits (by definition) and preserves the homotopy between circuits (by definition, again, $f$ maps homotopy matrices to homotopy matrices). At this point it is clear that $f_*$ is a group homomorphism.
\item It is clear, by the previous item and the fact that $f^{-1}$ is an NPP-function.
\item Since $g\circ f\simeq Id_X$, by Lemma \ref{lem:commutativediagram} we obtain the following commutative diagram

$$
\xymatrix{\pi_1(X,x)\ar[r]^{f_*}\ar[dr]_{u} & \pi_1(Y,f(x))\ar[d]^{g_*}\\
& \pi_1(X,(g(f(x))}
$$

Since $u$ is an isomorphism, it follows that $f_*$ is a monomorphism and $g_*$ is an epimorphism. Applying the same argument to $f\circ g$ we obtain that $g_*$ is a monomorphism and then it is an isomorphism. Now, since $g_*\circ f_*=u$ and both $g_*$ and $u$ are isomorphism, it follows that $f_*$ is an isomorphism, as required.
\end{enumerate}
\end{proof}

\section{Two applications}\label{se:firstapplications}

The theory that we are developing is, in some sense, between Graph Theory and Topology of Manifolds and this is why one can find many applications \emph{bringing up} results from Graph Theory to locally finite metric spaces or \emph{bringing down} results from Topology of Manifolds. What we mean, more specifically, is the following: suppose that one needs a version for locally finite metric space of a theorem for graphs. The idea is to decompose the metric space into its connected components. They are not graphs, but many arguments can be applied with little variations (basically because there is a notion of connectedness) and so one can re-prove variations of the claimed theorem for each of them and, then, put everything together back. This is the spirit behind our version for general metric spaces of the recent P.N.Jolissaint-Valette's inequality. On the other hand, since every path-connected component is a \emph{discretization} of the classical notion of path-connected component in General Topology, many arguments of Topology of Manifolds can be applied in each path-connected component (basically because there is a notion of continuity), giving versions of results for manifolds in locally finite metric spaces.
\\\\
Said this, one can have fun giving many applications of the theory. We have chosen to present two particular applications for specific reasons that will be explained at the beginning of their own subsection.

\subsection{The Jordan curve theorem in $\mathbb Z^2$}

The classical Jordan curve theorem states that a simple closed curve $\gamma$ in $\mathbb R^2$ separates $\mathbb R^2$ in two path-connected components, one bounded and one unbounded and $\gamma$ is the boundary of each of these components. In this section we want to prove the analogue result in $\mathbb Z^2$ with the Euclidean distance. The reason behind the choice of this application is that the Jordan curve theorem in $\mathbb Z^2$ is of interest in Digital Topology (see, for instance, \cite{Bo08} and reference therein), but it seems that the version that we are going to present is not known. Indeed, it is based on a particular definition of \emph{simplicity} of a curve that seems to be new. In fact, one is tempted to define a simple circuit in $\mathbb Z^2$ as a continuous circuit $x_0x_1\ldots x_{n-1}x_n$ such that the $x_i$'s are pairwise distinct for $i\in\{1,\ldots,n-1\}$. With this definition the Jordan curve theorem is false: consider the following continuous circuit:
$$
(0,0)(0,-1)(1,-1)(2,-1)(2,0)(2,1)(1,1)(1,2)(0,2)(-1,2)(-1,1)(-1,0)(0,0)
$$
it is \emph{simple} in the previous sense; it has fundamental group equal to $\mathbb Z$, but it does not separate the grid $\mathbb Z^2$ in two path-connected components. In some sense, this circuit behaves like the $8$-shape curve in $\mathbb R^2$, which is not simple. The discrete analog of simplicity is indeed something different. Let us fix some notation

\begin{itemize}
\item Let $(x_0,y_0)\in\mathbb Z^2$, we denote by $dB_1((x_0,y_0))$ the set
$$
\{(x_0,y_0+1),(x_0,y_0-1),(x_0-1,y_0),(x_0+1,y_0)\}
$$
which is called \textbf{discrete 1-boundary of $(x_0,y_0)$}.
\item Let $(x_0,y_0)\in\mathbb Z^2$, we denote by $dB_2((x_0,y_0))$ the set
$$
\{(x_0+1,y_0+1),(x_0-1,y_0+1),(x_0-1,y_0-1),(x_0+1,y_0-1)\}
$$
which is called \textbf{discrete 2-boundary of $(x_0,y_0)$}.
\end{itemize}

\begin{definition}\label{def:simplecurve}
A \textbf{simple curve} in $\mathbb Z^2$ is a continuous path $x_0x_1\ldots x_{n-1}x_n$ such that
\begin{itemize}
\item The $x_i$'s are pairwise distinct for $i\in\{1,\ldots,n-1\}$.
\item Whenever $x_i\in dB_2(x_j)$, for $j>i$, then $j=i+2$ and
$$
x_{i+1}\in dB_1(x_i)\cap dB_1(x_j)
$$
\end{itemize}
\end{definition}

This definition might seem \emph{ad hoc}, but it is indeed exactly how simplicity behaves in a discrete setting. Indeed, the main property of the classical notion of simple curve $\gamma$ in $\mathbb R^2$ is not the injectivity of the mapping $t\rightarrow\gamma(t)$, but the fact (implied by the injectivity), that $\gamma$ can be \emph{re-constructed} as follows: let $Int(\gamma)$ be the bounded path-connected component of $\mathbb R^2\setminus\gamma$ and $(x_0,y_0)\in Int(\gamma)$. Construct four points of $\gamma$ as follows: $(x_0^+,y_0)$ is the first point of the half-line $y=y_0$, with $x>x_0$, hitting $\gamma$; analogously, construct $(x_0^-,y_0),(x_0,y_0^+),(x_0,y_0^-)$. Then $\gamma$ is uniquely determined by the set of these points, when $(x,y)$ runs over the path-connected component of $\mathbb R^2\setminus\gamma$ containing $(x_0,y_0)$. Notice that this re-construction principle is false in $\mathbb R^2$ for closed curve having fundamental group equal to $\mathbb Z$ or for the $8$-shape curve. In the following, we shows that our definition of simple curve in $\mathbb Z^2$ gives exactly the property which allow to make this re-construction procedure.
\\\\
Let $\gamma$ be a simple curve. Observe that, up to homotopy, we can suppose that $\gamma$ does not contain any \emph{unit square}; namely, it does not contain four points of the shape $(x_0,y_0),(x_0+1,y_0),(x_0+1,y_0+1),(x_0+1,y_0)$. Let $(x,y)\in\mathbb Z^2\setminus\gamma$, we construct (at most) four points in the following way:
\begin{itemize}\label{it:extremal}
\item $(x_0^+,y_0)$ is the first point (if exists) where the horizontal half-line $y=y_0$, for $x>x_0$, hits $\gamma$;
\item $(x_0^-,y_0)$ is the first point (if exists) where the horizontal half-line $y=y_0$, for $x<x_0$, hits $\gamma$;
\item $(x_0,y_0^+)$ is the first point (if exists) where the vertical half-line $x=x_0$, for $y>y_0$, hits $\gamma$;
\item $(x_0,y_0^-)$ is the first point (if exists) where the vertical half-line $x=x_0$, for $y<y_0$, hits $\gamma$;
\end{itemize}

\begin{definition}\label{def:quasiinternal}
A point $(x_0,y_0)$ is called \textbf{quasi-internal to $\gamma$} if all four points defined above exist.
\end{definition}

\begin{definition}\label{def:internal}
A subset $C$ of $\mathbb Z^2$ is called \textbf{internal to $\gamma$} if
\begin{itemize}
\item Every $(x,y)\in C$ is quasi-internal to $\gamma$,
\item $C$ is called under the \emph{reconstruction procedure}; namely, if $(x_0,y_0)\in C$, then any point $(x,y_0)$, with $x\in[x_0^-,x_0^+]$, belongs to $C$ and also any point $(x_0,y)$, with $y\in[y_0^-,y_0^+]$, belongs to $C$.
\end{itemize}
In this context, the points of the shape $(x_0^\pm,y_0)$ and $(x_0,y_0^\pm)$, with $(x_0,y_0)\in C$, are called \textbf{extremal points of $C$ with respect to $\gamma$}. The set of extremal points of $C$ with respect to $\gamma$ is denoted by $Extr_\gamma(C)$.
\end{definition}

It is clear that one of the crucial part of the proof of the Jordan curve theorem in $\mathbb Z^2$ is the existence of an internal set. For now, let us suppose that it exists and prove the re-construction principle.
\\\\
Let $C\subseteq \mathbb Z_2$ be internal to $\gamma$ and suppose that for any $e\in Extr_\gamma(C)$ one of the following properties hold:
\begin{itemize}
\item $|dB_1(e)\cap Extr_\gamma(C)|=2$
\item $|dN_1(e)\cap Extr_\gamma(C)|=1$ and $|dB_2(e)\cap Extr_\gamma(C)|=1$
\item $|dN_1(e)\cap Extr_\gamma(C)|=0$ and $|dB_2(e)\cap Extr_\gamma(C)|=2$
\end{itemize}
If $e$ is a point of the second shape, we denote by $e'$ the unique point in $dB_2(e)\cap Extr_\gamma(C)$; if $e$ is a point of the third shape, we denote by $e',e''$ the two points in $dB_2(e)\cap Extr_\gamma(C)$. Under this hypothesis, we define the $\gamma$-completion of $C$ to be the set $\overline C^\gamma$ formed by the union between $C$ and all points $e'$ and $e',e''$. Now we can prove the \emph{reconstruction principle} in $\mathbb Z^2$.

\begin{lemma}\label{lem:reconstruction}
Let $\gamma$ be a simple circuit in $\mathbb Z^2$ which does not contain unit squares and let $C$ be a finite subset of $\mathbb Z^2$ which is internal to $\gamma$. Then $\gamma=\overline C^\gamma$
\end{lemma}

\begin{proof}
Fix $e=(x_0,y_0)\in Extr_\gamma(C)$. Since $(x_0,y_0)$ is extremal, we know that there is at least one point in $dN_1(x_0,y_0)$ which belongs to $C$. We can suppose that it is $(x_0+1,y_0)$. Since $\gamma$ is continuous and $(x_0,y_0)\in\gamma$, we know that there are at least two points in $dN_1(x_0,y_0)$ belonging to $\gamma$. Now, one has just several possibilities to study. Let us go through a couple of them (the remaining will be left to the reader).
\\\\
\textbf{First case:}\\
$(x_0,y_0-1),(x_0,y_0+1)\in\gamma$ and $(x_0+1,y_0-1),(x_0+1,y_0+1)\notin\gamma$. Since $C$ is closed under reconstruction, it follows that $(x_0+1,y_0-1),(x_0+1,y_0+1)\in C$ In this case it is evident that $(x_0,y_0-1)$ and $(x_0,y_0+1)$ are extremal and so we belong to the first of the previous three possibilities and there is no completion to do.\\\\

\textbf{Second case:}\\
$(x_0+1,y_0-1),(x_0+1,y_0+1)\in\gamma$. Observe that these two points belong also to $Extr_\gamma(C)$ and so we belong to the third possibility and we have $e'=(x_0,y_0+1)$ and $e''=(x_0,y_0-1)$. Note that $e',e''\in\gamma$, since $\gamma$ is simple.\\\\

\textbf{Third case:}\\
$(x_0,y_0-1),(x_0+1,y_0+1)\in\gamma$ and $(x_0+1,y_0-1)\notin\gamma$. Since $C$ is closed under reconstruction, it follows that $(x_0+1,y_0-1)\in C$. So we are in the second possibility and so $e'=(x_0,y_0+1)$. Notice that $e'\in\gamma$, since $\gamma$ is simple.
\\\\
And so on tediously. Now, observe that we have proved that for any $e\in Extr_\gamma(C)$, we can find a piece of $\gamma$ in $dB_1(e)\cup dB_2(e)$ and this piece of $\gamma$ belongs, by construction, to $\overline C^\gamma$. Since $C$ is finite, this procedure gives rise a sub-circuit of $\gamma$, which has to coincide with $\gamma$, since the latter is simple and does not contain unit squares.
\end{proof}

\begin{theorem}\label{th:jordanbrouwer}
Let $\gamma$ be a non-constant simple circuit in $\mathbb Z^2$ not containing squares\footnote{This assumption is, in some sense, without loss of generality, because we can always assume it \emph{up to homotopic equivalence}.}. Then $\mathbb Z^2\setminus\gamma$ has exactly two path-connected components, denoted by $Ext(\gamma)$ and $Int(\gamma)$ and the following properties are satisfied:
\begin{itemize}
\item $Int(\gamma)$ is finite;
\item $Ext(\gamma)$ is infinite;
\item $\gamma=\overline{Extr_\gamma(Int(\gamma))}^\gamma$
\end{itemize}
\end{theorem}

\begin{proof}
For the convenience of the reader, we divide the proof in several steps.
\\\\
\textbf{First step: definition of $Int(\delta)$ and $Ext(\delta)$.}\\
Let $\gamma=c_0c_1\ldots c_{n-1}c_n$ and $\Gamma=\{c_0,\ldots c_{n-1}\}$. Fix the following integers:
$$
x^-=\min\{x:(x,y)\in\Gamma\}
$$
$$
x^+=\max\{x:(x,y)\in\Gamma\}
$$
$$
y^-=\min\{y:(x,y)\in\Gamma\}
$$
$$
y^+=\max\{y:(x,y)\in\Gamma\}
$$

Now, for any integer $k\geq0$ and for any integer $y\in[y^--k,y^++k]$ consider the continuous path
$$
\gamma_{y,k}=(x^--k,y)(x^--k+1,y)\ldots(x^++k-1,y)(x^++k,y)
$$
and denote by $\Gamma_{y,k}$ the set $\{(x^--k,y),(x^--k+1,y),\ldots,(x^++k-1,y),(x^++k,y)\}$

Let $(x_1,y^-)(x_2,y^-)\ldots(x_r,y^-)$ be the \emph{first} continuous path obtained by the intersection between $\Gamma$ and the \emph{horizontal line} $y=y^-$; i.e. $x_1$ is the minimal $x$ such that $(x,y^-)\in\Gamma$. Observe that $r\neq1$. Indeed, if it was $r=1$, then $\gamma$ would not have been simple. It follows that $r\geq2$ and therefore $\Gamma$ contains $(x_1,y^-),(x_1+1,y^-)$ and $(x_1,y^-+1)$. Hence, $(x_1+1,y^-+1)\notin\Gamma$ (since $\gamma$ does not contain unit squares). We define $Int(\gamma)$ to be the path-connected component of $\mathbb Z^2\setminus\gamma$ containing $(x_1+1,y^-+1)$. Of course, we now define $Ext(\gamma)$ to be the complementary of $Int(\delta)\cup\gamma$.
\\\\
\textbf{Second step: $Int(\gamma)$ is finite and $Ext(\gamma)$ is infinite.}\\
By definition $Int(\gamma)$ is non-empty. In order to prove that it is finite and that $Ext(\gamma)$ is infinite, we prove that $Int(\gamma)$ is contained in the rectangle $R=[x^-,x^+]\times[y^-,y^+]$. Let us denote by $(z_0,w_0)$ the point $(x_1+1,y^-+1)$ that we have found in the previous step. Let $(z_1,w_1)\in dN_1((z_0,y_0))\cap Int(\gamma)$. There are two possibilities: $(z_1,w_1)=(z_0+1,w_0)$ or $(z_1,w_1)=(z_0,w_0+1)$. In each case, one has $z_1\geq x^-$ and $w_1\geq y^-$. Now consider the vertical line $z=z_1$ and suppose, by contradiction, that it intersects $\gamma$ only in $(z_1,w_1-1)$. Since $\gamma$ is simple, it necessarily follows that either $\gamma\subseteq\{(x,y):x<z_1\}$ or $\gamma\subseteq\{(x,y):x>z_1\}$. But also in these cases, one contradicts the fact that $\gamma$ is simple. It follows that the line $x=z_1$ intersects $\gamma$ in another point $(x_1,y)$ and then $w_1\leq y\leq y^+$. It follows that $w_1\leq y_+$. Analogously, taking the horizontal line $y=w_1$, one gets $z_1\leq x^+$. Hence $(z_1,w_1)\in R$. Now repeat the argument for any $(z_2,w_2)\in dN_1((z_1,w_1))\cap Int(\gamma)$ and we get the same. Since $Int(\gamma)$ is, by definition, the path-connected component containing $(z_0,w_0)$ we obtain indeed that $Int(\gamma)\subseteq R$.
\\\\

\textbf{Third step: $Int(\gamma)$ and $Ext(\gamma)$ are path-connected.}\\
$Int(\gamma)$ is path-connected by definition. Let us prove that also $Ext(\gamma)$ is path-connected. Let $(x_0,y_0)\in Ext(\gamma)$, it suffices to find a continuous path starting from $(x_0,y_0)$ and going outside $R$ without hitting $\gamma$. We proceed by making continuously the reconstruction procedure:
\begin{itemize}
\item We run continuously along the half-line $x=x_0$, with $y>y_0$. If we go outside $R$ without hitting $\delta$, the proof is over; otherwise, let $(x_0,y_0^+)\in\delta$ be the first point where we hit $\delta$. At the same way, we construct $(x_0,y_0^-)$, running continuously along the half-line $x=x_0$, with $y<y_0$; then, we construct $(x_0^-,y_0)$ and finally $(x_0^+,y_0)$. In this way we construct a continuous path $\delta_1$.

\item For any point in $\delta_1$, we repeat the argument in the first step, constructing a continuous path $\delta_2$.

\item We repeat the argument, until possible.
\end{itemize}

Now, suppose that the procedure ends after a finite number of steps, we want to obtain a contradiction by showing that $(x_0,y_0)\in Int(\gamma)$. Indeed, apply the reconstruction procedure first starting from $(x_0,y_0)$ and then starting from $(z_0,w_0)$ (it is the same point as the previous step of the proof). We obtain two internal sets $C$ and $C'$ whose extremal points define, thanks to the reconstruction principle, the same circuit. In particular, it follows that $C\cap C'\neq\emptyset$. Let $(x,y)\in C\cap C'$ and let $C''$ be the set obtained by making the reconstruction procedure starting from $(x,y)$. It is clear that $C=C''=C'$ and, in particular, $C=C'$. It follows that $(x_0,y_0)$ is in the same path-connected component of $\mathbb Z^2\setminus\gamma$ containing $(z_0,w_0)$, which is indeed, by definition, $Int(\gamma)$. Hence, the reconstruction procedure never ends and so it is always possible to add points. It follows that we have a continuous path starting in $(x_0,y_0)$ and ending arbitrarily far away.
\\\\
\textbf{Fourth step: $\gamma=\overline{Extr_\gamma(Int(\gamma))^\gamma}$}\\
This is now clear, from the proof of the previous steps and from the reconstruction principle.
\end{proof}

We conclude this subsection with some remarks.

\begin{remark}
\begin{enumerate}
\item It seems to be likely that $\pi(\gamma)=\mathbb Z$, $\pi_1(Ext(\gamma))=\mathbb Z$ and $Int(\gamma)$ is contractible, i.e. homotopic equivalent to one point.
\item It is likely that some weaker version of the Jordan curve theorem holds for continuous circuits, possibly not simple, having fundamental group equal to $\mathbb Z$. In fact, it is well possible that such curves separate the grid $\mathbb Z^2$ in \textbf{at least two} path-connected components.
\end{enumerate}
\end{remark}

\subsection{P.N.Jolissaint-Valette's inequality for finite metric spaces}\label{se:jolivale}

In a lecture at Lausanne, Alain Valette presented a new and interesting inequality for the $\ell_p$-distortion of a finite connected graph, proved by himself and his co-author Pierre-Nicolas Jolissaint (see \cite{Jo-Va11}, Theorem 1). At that point one of the auditors came up with the question whether there is some version for general finite metric spaces of that inequality. Here we want to propose such a more general version. The author thanks Pierre-Nicolas Jolissaint for reading an earlier version of this section and suggesting many improvements.
\\\\
Let $(X,d)$ be a finite metric space and let $X_1,\ldots X_n$ be the partition of $X$ in path-connected components. The basic idea is clearly to apply P.N.Jolissaint-Valette's inequality on each of them, but unfortunately this application is not straightforward, since a path-connected component might not look like a graph (think, for instance, to the space $[-n,n]^2\setminus\{(0,0)\}$ equipped with the metric induced by the standard embedding into $\mathbb R^2$). So we have to be a bit careful to apply P.N.Jolissaint-Valette's argument.

\begin{remark}\label{rem:normalform}
Since the $\ell_p$-distortion does not depend on rescaling the metric and since we are going to work on each path-connected component separately, we can suppose that each $X_i$ is in normal form.
\end{remark}

Since we are going to work on a fixed path-connected component, let us simplify the notation assuming directly that $X$ is finite path-connected metric space in normal form. At the end of this section it will be easy to put together all path-connected components.

Let $x,y\in X$, $x\neq y$ and let $x_0x_1\ldots x_{n-1}x_n$ be a continuous path joining $x$ and $y$ of minimal length $n$. Denote by $s$ the floor of $d(x,y)$, i.e. $s$ is the greatest positive integer smaller than or equal $d(x,y)$. Notice that $s\geq1$, since $X$ is in normal form. Denote by $\mathcal P(x,y)$ the set of coverings $P=\{p_1,\ldots,p_s\}$ of the set $\{0,1,\ldots,n\}$ such that
\begin{itemize}
\item If $a\in p_i$ and $b\in p_{i+1}$, then $a\leq b$,
\item the greatest element of $p_i$ equals the smallest element of $p_{i+1}$.
\end{itemize}
We denote by $p_i^-$ and $p_i^+$ respectively the smallest and the greatest element of $p_i$.

Now we introduce the following set

\begin{align}\label{eq:edgesmetric}
E(X)=\left\{(e^-,e^+)\in X\times X: \exists x,y\in X, p\in\mathcal P(x,y):e^-=p_i^-, e^+=p_i^+\right\}
\end{align}

\begin{remark}\label{rem:metricedgesvsgraphedges}
If $X=(V,E)$ is a finite connected graph equipped with the shortest path metric, then $E(X)=E$. Indeed in this case $s=n$ and so the only coverings belonging to $\mathcal P(x,y)$ have the shape $p_i=\{x_{i-1},x_i\}$, where the $x_i$'s are taken along a shortest path joining $x$ and $y$.
\end{remark}

We define a metric analog of the $p$-spectral gap: for $1\leq p<\infty$, we set

\begin{align}\label{eq:spectralgap}
\lambda_1^{(p)}=\inf\left\{\frac{\sum_{e\in E(X)}|f(e^+)-f(e^-)|^p}{\inf_{\alpha\in\mathbb R}\sum_{x\in X}|f(x)-\alpha|^p}\right\}
\end{align}

where the infimum is taken over all functions $f\in\ell^p(X)$ which are not constant.

\begin{lemma}\label{lem:jolivalette}
Let $(X,d)$ be a finite path-connected metric space in normal form.
\begin{enumerate}
\item For any permutation $\alpha\in Sym(X)$ and $F:X\rightarrow\ell^p(\mathbb N)$, one has
$$
\sum_{x\in X}||F(x)-F(\alpha(x))||_p^p\leq2^p\sum_{x\in X}||F(x)||_p^p
$$
\item For any bi-lipschitz embedding\footnote{For the convenience of the reader we recall that a bi-lipschitz embedding, in this context, is a mapping $F:X\rightarrow\ell^p$ such that there are constants $C_1,C_2$ such that for all $x,y\in X$ one has
$$
C_1d(x,y)\leq d_p(F(x),F(y))\leq C_2d(x,y)
$$
where $d_p$ stands for the $\ell^p$-distance. It is clear that $F$ is injective and so we can consider $F^{-1}:F(X)\rightarrow X$. So we can define
$$
||F||_{Lip}=\sup_{x\neq y}\frac{d_p(F(x),F(y))}{d(x,y)}
$$
and
$$
||F^{-1}||_{Lip}=\sup_{x\neq y}\frac{d(x,y)}{d_p(F(x),F(y))}
$$
The product $||F||_{Lip}||F^{-1}||_{Lip}$ is called \textbf{distortion of $F$} and denoted by $Dist(F)$.} $F:X\rightarrow\ell^p(\mathbb N)$, there is another bi-lipschitz embedding $G:X\rightarrow\ell^p(\mathbb N)$ such that $||F||_{Lip}||F^{-1}||_{Lip}=||G||_{Lip}||G^{-1}||_{Lip}$ and
$$
\sum_{x\in X}||G(x)||_p^p\leq\frac{1}{\lambda_1^{(p)}}\sum_{e\in E(X)}||G(e^+)-G(e^-)||_p^p
$$
\end{enumerate}
\end{lemma}
\begin{proof}
\begin{enumerate}
\item This proof is absolutely the same as the proof of Lemma 1 in \cite{Jo-Va11}.
\item Observe that the construction of $G$ made in \cite{Gr-No10} is purely algebraic and so we can apply it. The inequality just follows from the definition of $\lambda_1^{(p)}$.
\end{enumerate}
\end{proof}

Now we have to prove a version for metric spaces of a lemma already proved by Linial and Magen for finite connected graph (see \cite{Li-Ma00}, Claim 3.2). We need to introduce a number that measures how far is the metric space to be graph. We set

\begin{align}\label{eq:dis}
d(X)=\max_{e\in E(X)}d(e^-,e^+)
\end{align}

We told that $d(X)$ measures how far $X$ is far from being a graph since it is clear that $d(X)=1$ if and only if $X$ is a finite connected graph, in the sense that the metric of $X$ is exactly the length of the shortest path connected two points. Indeed, in one sense, it is trivial: if $X$ is a finite graph, then $E(X)=E$ (by Remark \ref{rem:metricedgesvsgraphedges}) and $d(X)=1$, since the space is supposed to be in normal form. Conversely, suppose that $d(X)=1$, choose two distinct points $x,y\in X$ and let $x_0x_1\ldots x_{n-1}x_n$ be a continuous path of minimal length such that $x_0=x$ and $x_n=y$. Of, course $s\leq d(x,y)\leq n$. So, it suffices to prove that $s=n$. In order to do that, suppose that $s<n$ and observe that every $p\in\mathcal P(x,y)$ would contain a $p_i$ with at least three points $p_{i}^-=x_{i-1},x_i,x_{i+1}=p_i^+$ (we suppose that they are exactly three, since the general case is similar. Since $X$ is in normal form, it follows that $d(x_{i-1},x_i)=d(x_i,x_{i+1})=1$. Now, suppose that $d(X)=1$, it follows that also $d(x_{i-1},x_{i+1})=1$ and then the path $x_0\ldots x_{i-1}x_{i+1}\ldots x_n$ is still a continuous path connecting $x$ with $y$, contradicting the minimality of the length of the previous path.

\begin{lemma}\label{lem:linialmagen}
Let $(X,d)$ be a finite path-connected metric space in normal form and $f:X\rightarrow\mathbb R$. Then
\begin{align}
\max_{x\neq y}\frac{|f(x)-f(y)|}{d(x,y)}\leq\max_{e\in E(X)}|f(e^+)-f(e^-)|\leq d(X)\max_{x\neq y}\frac{|f(x)-f(y)|}{d(x,y)}
\end{align}
\end{lemma}

\begin{proof}
Let us prove only the first inequality, since the second will be trivial \emph{a posteriori}. Let $x,y\in X$ where the maximum in the left hand side is attained and let $x_0x_1\ldots x_{n-1}x_n$ be a continuous path of minimal length connecting $x$ with $y$. Let $p\in\mathcal P(x,y)$ and let $k$ be an integer such that
$$
|f(p_k^+)-f(p_k^-)|\geq|f(p_i^+)-f(p_i^-)|
$$
for all $i$. It follows
$$
|f(p_k^+)-f(p_k^-)|=\frac{s|f(p_k^+)-f(p_k^-)|}{s}\geq\frac{\sum_{i=1}^s|f(p_i^+)-f(p_i^-)|}{s}\geq
$$
Now we use the fact that the covering $p$ is made exactly by $s$ sets. It follows that $x$ and $y$ belong to the union of the $p_i$'s and we can use the triangle inequality and obtain
$$
\geq\frac{|f(x)-f(y)|}{s}\geq\frac{|f(x)-f(y)|}{d(x,y)}
$$
\end{proof}

The following result was already proved for finite connected graph in \cite{Jo-Va11}, Theorem 1 and Proposition 3.

\begin{theorem}\label{th:jolivale}
Let $(X,d)$ be a finite path-connected metric space in normal form. For all $1\leq p<\infty$, one has
\begin{align}
c_p(X)\geq\frac{D(X)}{2d(X)}\left(\frac{|X|}{|E(X)|\lambda_1^{(p)}}\right)^{\frac{1}{p}}
\end{align}
where
\begin{align}\label{eq:displacement}
D(X)=\max_{\alpha\in Sym(X)}\min_{x\in X}d(x,\alpha(x))
\end{align}

and

\begin{align}\label{eq:distortion}
c_p(X)=\inf\left\{||F||_{Lip}||F^{-1}||_{Lip}, F:X\rightarrow\ell^p(\mathbb N) \text{ bi-lipschitz embedding }\right\}
\end{align}
is the $\ell^p$-distortion of $X$.
\end{theorem}

\begin{proof}
Let $G$ be a bi-lipschitz embedding which verifies the second condition in Lemma \ref{lem:jolivalette} and let $\alpha$ be a permutation of $X$ without fixed points. Let $\rho(\alpha)=\min_{x\in X}d(x,\alpha(x))$. One has
$$
\frac{1}{||G^{-1}||_{Lip}^p}=\min_{x\neq y}\frac{||G(x)-G(y)||_p^p}{d(x,y)^p}\leq\min_{x\in X}\frac{||G(x)-G(\alpha(x))||_p^p}{d(x,\alpha(x))^p}\leq
$$
$$
\leq\frac{1}{\rho(\alpha)^p}\min_{x\in X}||G(x)-G(\alpha(x))||_p^p\leq
$$
$$
\leq\frac{1}{\rho(\alpha)^p|X|}\sum_{x\in X}||G(x)-G(\alpha(x))||_p^p\leq
$$
Now apply the first statement of Lemma \ref{lem:jolivalette}:
$$
\leq\frac{2^p}{\rho(\alpha)^p|X|}\sum_{x\in X}||G(x)||_p^p\leq
$$
Now apply the second statement of Lemma \ref{lem:jolivalette}:
$$
\leq\frac{2^p}{\rho(\alpha)^p|X|\lambda_1^{(p)}}\sum_{e\in E(X)}||G(e^+)-G(e^-)||_p^p\leq
$$
$$
\leq \frac{2^p|E(X)|}{\rho(\alpha)^p|X|\lambda_1^{(p)}}\max_{e\in E(X)}||G(e^+)-G(e^-)||_p^p\leq
$$
Now apply Lemma \ref{lem:linialmagen}:
$$
\leq\frac{2^pd(X)^p|E(X)|||G||_{Lip}^p}{\rho(\alpha)^p|X|\lambda_1^{(p)}}
$$
Now recall the definition in Equations \ref{eq:distortion} and \ref{eq:displacement} and just re-arrange the terms to get the desired inequality.
\end{proof}

Notice that $d(X)\geq1$ and so the inequality gets worse when the metric space is not a graph. It suffices to write down some pictures to understand that it is likely that one can replace $d(X)$ with a smaller constant. For the moment we are not interested in finding the best inequality, but just in showing how to apply some notions that we have introduced in order to extend results from the setting of (locally) finite graphs to (locally) finite metric spaces. Analogously, now we present a version for general metric spaces of P.N.Jolissaint-Valette's inequality which is certainly not the best possible, because one can try to make it better studying how the different path-connected components are related among themselves.

\begin{corollary}\label{cor:jolivale}
Let $X$ be a metric space such that every path-connected component $X_i$ is finite. One has
\begin{align}
c_p(X)\geq\sup_i\frac{D(X_i)}{2d(X_i)}\left(\frac{|X_i|}{|E(X_i)|\lambda_1^{(p)}}\right)^{\frac{1}{p}}
\end{align}
\end{corollary}

\begin{proof}
Just observe that if $X_i$ is a partition of $X$ then $c_p(X)\geq\sup_ic_p(X_i)$.
\end{proof}

\section{NPP-isomorphisms}\label{se:definitions}

In Sec.\ref{suse::discretecontinuity} we got very close to a new notion of isomorphism between two locally finite metric spaces: the notion of NPP-local-isomorphism is already a good notion, but it is \emph{too local} to capture global behaviors. In this section we introduce the notion of NPP-isomorphism, which is just a global version of NPP-local isomorphisms. If we were interested only in locally finite metric spaces, the definitions would have been much easier. In light of the property SN, that seems to be of interest for general metric spaces, in this section we work on general metric spaces.

\subsection{The discrete neighborhood of a bounded set}\label{suse:discrete}

In this first subsection we introduce the notion of \emph{discrete $k$-neighborhood} of a bounded set, as a generalization of the discrete 1-neighborhood of a singleton, introduced in Definition \ref{def:1discrete}. The basic idea is that we want a notion able to capture at the same time information about continuity and discreteness on the neighborhood of a bounded subset $A$ of $X$. We try do that introducing a notion that gets trivial when the metric space near $A$ is \emph{intuitively} continuous.

\begin{definition}\label{def:completechain}
Let $P\subseteq[0,\infty)$, with $0\in P$. A complete chain in $P$ is a finite subset $p_1,p_2,\ldots,p_n$ such that
\begin{enumerate}
\item $p_1=0$
\item $p_i<p_j$, for all $i<j$
\item if $p\in P$ is such that $p_i\leq p\leq p_{i+1}$, then either $p=p_i$, or $p=p_{i+1}$.
\end{enumerate}
\end{definition}

\begin{definition}\label{def:simpleboundary}
Let $(X,d)$ be any metric space, let $A$ be a bounded subset of $X$ and let $k$ be a nonnegative integer. The discrete $k$-neighborhood of $A$ is the set of elements $x\in X$ (resp. $x\in X\setminus A$) such that there is a finite sequence $x_0,x_1,\ldots,x_l=x$ in $X$ such that
\begin{enumerate}
\item $x_0\in A$
\item $l\leq k$
\item $d(x_0,x_0)<d(x_0,x_1)<d(x_0,x_2)<\ldots<d(x_0,x_l)$ is a complete chain in $P=\{d(y,A),y\in X\}$
\end{enumerate}
\end{definition}

Notice that there are no relations of inclusion between $dN_k$ and $cN_{k-1}$ (recall that $cN_k(A)=\{x\in X:d(x,A)\leq k\}$. For instance
\begin{itemize}
\item Take $A=(0,1)$ inside $X=\mathbb R$ equipped with the Euclidean metric. Then $cN_1(A)=(-1,2)$ and $dN_2(A)=[0,1]$.
\item Take $A=\{0\}$ inside $X=\{0,2\}$ equipped with the Euclidean metric. Then $cN_1(A)=\{0\}$ and $dN_2(A)=\{0,2\}$.
\end{itemize}
The first example, in particular, shows the meaning of our mysterious sentence: \emph{the discrete boundary is something that gets trivial when the space is intuitively continuous}. So let us define a notion of global continuity of a space which will be also useful in the sequel to give examples of spaces with property SN (see Proposition \ref{prop:example}).

\begin{definition}\label{def:discreteboundary}
The discrete $k$-boundary of a bounded subset $A$ of $X$ is $dB_k(A)=dN_k(A)\setminus A$.
\end{definition}

\begin{definition}\label{def:continuous}
A metric space $(X,d)$ is said to be \textbf{continuous} if for all open and bounded subsets $A$ of $X$ and for all $k\in\mathbb N$, one has $dB_k(A)=\overline A\setminus A$.
\end{definition}

\subsection{Boundary and NPP-isomorphisms}\label{suse:npp}

This subsection is devoted to two important definitions. The first is the definition of boundary that we are interested in. Since we want to capture at the same time continuity and discreteness, it is natural to use two terms, one looking at the discrete part, the other at the continuous part. Let us fix a piece of notation. Let $A\subseteq X$ and let $k\geq1$ be an integer. First of all, set $M=\max\{d(x,A),x\in dB_k(A)\}$ and then
$$
\alpha(k)=\left\{
            \begin{array}{ll}
              k, & \hbox{if $M=0$;} \\
              M, & \hbox{if $M\neq0$.}
            \end{array}
          \right.
$$
\begin{definition}\label{def:boundary}
Let $k\geq1$ be integer and $A\subseteq X$. The \textbf{$k$-boundary} of $A$ is the set
$$
B_k(A)=dB_k(A)\cup\bigcup_{0\leq\alpha\leq\alpha(k)-1}cB_\alpha(A)
$$
where $\alpha$ runs over the real numbers (not just over the integers).
\end{definition}

\begin{lemma}\label{lem:discrete}
If $(X,d)$ is locally finite and $A$ is finite and non empty, then
$$
B_k(A)=dB_k(A)
$$
\end{lemma}

\begin{proof}
By definition, there is no continuous contribution; namely, the continuous contribution reduces to the discrete distribution.
\end{proof}

Observe that this lemma would have been false, without introducing the parameter $\alpha(k)$, since otherwise there might have been some extra-contributions from the discrete part.
\\\\
The second definition is the notion of equivalence that is of our interest.

\begin{definition}\label{definition:nppdiscrete}
The metric space $(X,d_X)$ is NPP-embeddable (resp. isomorphic) into (resp. to) the metric space $(Y,d_Y)$ if and only if there is an injective (resp. bijective) map $f:X\rightarrow Y$ such that if (resp. and only if) $x_1,\ldots,x_n\in X$ give rise to a complete chain
$$
d_X(x_1,x_1)<d_X(x_1,x_2)<\ldots<d_X(x_1,x_n)
$$
inside $P=\{d_X(x_1,x), x\in X\}$, then $f(x_1),\ldots f(x_n)$ give rise to a complete chain
$$
d_Y(f(x_1),f(x_1))<\ldots<d_Y(f(x_1),f(x_n))
$$
inside $P=\{d_Y(f(x_1),y),y\in Y\}$
\end{definition}

This condition becomes particularly simple in the case of locally finite metric space, thanks to Lemma \ref{lem:discrete}.

\begin{proposition}\label{prop:nppdiscrete}
A bijective map between locally finite metric spaces is a NPP-isomorphisms if and only for all finite non-empty subset $A\subseteq X$, one has
$$
f(dN_k(A))=dN_k(f(A))
$$
\end{proposition}

This proposition makes clear why this notion is the global version of NPP-local-isomorphisms.

\subsection{Examples of NPP-embeddings}\label{se:examples}

We have just introduced \emph{another} way to embed a metric space into another metric space. Indeed, the theory of embedding a (locally finite) metric space into a (possibly well understood) metric space has a long and interesting story over the 20th century and still alive. At the very beginning of the story, people were interested in isometric embeddings, but at some point it became clear that it is very rare to find an isometry between two metric spaces of interest. At that point people got interested in various approximate embeddings: quasi-isometric embeddings, bi-lipschitz embeddings, coarse embeddings, uniform embeddings and so on. In particular, the theory became very popular after two breakthrough papers, which showed unexpected correlations with other branches of Mathematics or, even, other sciences: Linial, London and Rabinovich found a link between this theory and some problems in theoretical computer science (see \cite{Li-Lo-Ra95}); Guoliang Yu showed a link between this theory, geometric group theory and K-theory of $C^*$-algebras (see \cite{Yu00}). The common paradigm of all these embeddings is to forget, more or less brutely, the local structure of the metric space in order to try to understand the global structure. As we said at the very beginning of this paper, because of our motivating problems, we are interested in the local structure of a metric space and this is why we cannot use one of those embeddings.

\begin{remark}
Let us observe explicitly that the notion of NPP-embedding is different from all other notions used to study metric spaces.

\begin{itemize}
\item It is clear that an isometry is also an NPP-embedding. The converse does not hold: consider, inside the real plane with the Euclidean metric, the subspace $X=(0,1)\cup\{(n,0),n\in\mathbb N, n\geq2\}$. It is clear that $X$ is NPP-isomorphic to $Y=\{(n,0), n\in\mathbb N, n\neq1\}$, but there is no isometry from $X$ to any subspace of $\mathbb Z$ (simply because of the irrational distances).
\item Take two parallel copies of $\mathbb Z$ inside the Euclidean plane, for instance: $X=\{(0,x),x\in\mathbb Z\}\cup\{(1,x), x\in\mathbb Z\}$. Then $X$ is quasi-isometric\footnote{Recall that two metric spaces $(X,d_X)$ and $(Y,d_Y)$ are called quasi-isometric if there is a map $f:X\rightarrow Y$ and two constants $\lambda\geq1$ and $C\geq0$ such that
\begin{enumerate}
\item $\lambda^{-1}d(x_1,x_2)-C\leq d(f(x_1),f(x_2))\leq\lambda d(x_1,x_2)+C$, for all $x_1,x_2\in X$
\item $d(y,f(X))\leq C$, for all $y\in Y$
\end{enumerate}
} and coarsely equivalent\footnote{Let $(X,d_X)$  and $(Y,d_Y)$ be two metric spaces. A map $f:X\rightarrow Y$ is called \textbf{coarse embedding} if there are two non-decreasing functions $\rho_\pm:\mathbb R_+\rightarrow\mathbb R_+$ such that
$$
\rho_-(d_X(x,y))\leq d_Y(f(x),f(y))\leq\rho_+(d_X(x,y))
$$
The spaces $X$ and $Y$ are called \textbf{coarsely equivalent} if there are two coarse embeddings $f:X\rightarrow Y$ and $g:Y\rightarrow X$ and a constant $C\geq0$ such that
\begin{enumerate}
\item $d_X(x,g(f(x)))\leq C$, for all $x\in X$
\item $d_Y(y,f(g(y)))\leq C$, for all $y\in Y$
\end{enumerate}
}
to a single copy of $\mathbb Z$. But it is clear that there are no NPP-isomorphisms between $X$ and a single copy of $\mathbb Z$. Indeed, the discrete 1-boundary of a singleton in $X$ contains three points and the discrete 1-boundary of a singleton in $\mathbb Z$ contains two points.
\end{itemize}

\end{remark}

Now we want to propose a general example of NPP-isomorphism, which generalizes the situation when $(X,d)$ is a connected graph equipped with the shortest-path distance.
Indeed, given a metric space, we can always represent it as a graph, joining any pair of points $(x,y)$ with an edge labeled with the number $d(x,y)$. This sounds particularly good if the space is a connected graph, because $d(x,y)\in\mathbb N$, but in general we are not allowed to do such a construction, since the nature of the problem exclude the use of $d(x,y)$ (see the problems discussed in Sec. \ref{suse:otherapplications}). so the aim of what follows is to find conditions that guarantee the existence of a canonical NPP-isomorphism between a locally finite metric space $(X,d)$ and a TS-graph\footnote{A weighted graph $X=(V,V^2)$ (then any pair of vertex is joined by an edge) is called \textbf{traveling salesman graph (TS graph)}, if any edge $\{x,y\}$ is labeled by a non-negative integer $d(x,y)$, which gives rise to a metric on $V$.} whose distances are natural numbers.

The proof is based on the first intuitive procedure to make a map: start from a base point $x_0\in X$, look at the nearest points (i.e., look at $dN_1(\{x_0\})$) and draw a point for each of them. We obtain a symbolic representation of some region $R_1$ of $X$. Now, increase the knowledge about the space looking at $dN_2(\{x_0\})$ and draw a point for any new point that we find; and so on. Unfortunately, without some assumptions, this procedure faces some difficulties. The simplest assumptions we can find says basically that the discrete neighborhoods behave like the balls. It is likely that one can find more general results, but, at the same time, it is likely that in general there is no NPP-isomorphism between a locally finite metric space and a TS-graph whose distances are natural numbers.

\begin{definition}\label{def:graphtype}
A locally finite metric space is called \textbf{graph-type} if
\begin{enumerate}
\item $$x\in dB_k(y)\setminus dB_{k-1}(y) \Leftrightarrow y\in dB_k(x)\setminus dB_{k-1}(x)$$
\item $$x\in dN_h(y) \text{ and } z\in dN_k(y) \text{ , imply } x\in dN_{h+k}(z)$$
\end{enumerate}
\end{definition}

\begin{theorem}\label{th:NPPisomorphism1}
Any graph-type space $X$ is NPP-isomorphic to a canonical TS-graph, which is called \textbf{symbolic graph of $X$}.
\end{theorem}

\begin{proof}

Fix a base point\footnote{It will be clear at the end of the construction that it does not depend on the base point (it will follow by the definition of graph-type metric spaces).} $x_0\in X$ and construct a weighted graph as follows:

\begin{itemize}
\item For any point $y\in dB_1(x_0)$, draw an edge $\{x,y\}$ labeled with the number 1.
\item
\begin{itemize}
\item For any point $x\in dB_1(x_0)$, consider any $y\in dB_1(x)$. If there is no edge joining $x$ and $y$, draw an edge $\{x,y\}$ labeled with the number $1$.
\item For any point $x\in dB_1(x_0)$, join $x$ to any point $y\in dB_2(x)$ and label this edge with the number 2.
\end{itemize}
\item
\begin{itemize}
\item For any point $x\in dB_2(x_0)$, consider any $y\in dB_1(x)$. If there is no edge joining $x$ and $y$, draw an edge $\{x,y\}$ labeled with the number $1$.
\item For any point $x\in dB_2(x_0)$, consider any $y\in dB_2(x)$. If there is no edge joining $x$ and $y$, draw an edge $\{x,y\}$ labeled with the number $2$.
\item For any point $x\in dB_2(x_0)$, join $x$ to any point $y\in dB_3(x)$, with an edge labeled with the number $3$.
\end{itemize}
\item and so on.
\end{itemize}
It is clear that this construction produces a graph with vertex set $X$. Let us denote by $Symb(X)$ this graph. We have to check that the labels give rise to a metric on $Symb(X)$, that we denote by $Symb(d)$. Obviously, this metric is defined as follows:
$$
Symb(d)(x,y)=\left\{
          \begin{array}{ll}
            0, & \hbox{if $x=y$;} \\
            \ell(x,y), & \hbox{otherwise.}
          \end{array}
        \right.
$$
where $\ell(x,y)$ is the label of the (unique by construction) edge joining $x$ and $y$. The construction implies that $\ell(x,y)$ is the unique integer $h\geq1$ such that $y\in dB_h(x)\setminus dB_{h-1}(x)$. From the properties in Definition \ref{def:graphtype}, immediately follows that the symbolic metric is indeed a metric on the symbolic graph. It remains to prove that the identity map is an NPP-isomorphism between the two metrics. Let $x_0,x_1,\ldots,x_n\in X$ giving rise to a complete chain
$$
d(x_0,x_0)<d(x_0,x_1)<\ldots
$$
in $P=\{d(x_0,x),x\in X)$. It follows that $x_j\in dB_j(x_0)\setminus dB_{j-1}(x_0)$, for all $j=1,\ldots,n$. Hence $Symb(d)(x_0,x_j)=j$ which is a complete chain in $\{Symb(d)(x_0,x),x\in X\}$. Analogously, one gets the converse.
\end{proof}

Theorem \ref{th:NPPisomorphism1} is, in some sense, a starting theorem. It uses some natural conditions in order to guarantee the (almost) best result: to make a symbolic map using only integer numbers. In Sec. \ref{suse:otherapplications} we have already discussed that there are basically two different ways to develop this result: relax the hypothesis in order to have a symbolic map which uses only rational numbers with a fixed number of digits; strengthen/change the hypothesis in order to have a \emph{finite-dimensional} symbolic map. The point that makes the problem sometimes difficult is that, in the real-life problems, one needs both the condition.

\section{The isoperimetric constant of a locally finite space}\label{se:isoperimetric}

Now that we have a new notion of isomorphism between locally finite metric spaces, we want to construct some invariants. We already have an NPP-invariant; namely the fundamental group, which comes from the \emph{continuous world of manifolds thanks to a procedure of discretization of continuity}. In this section we want to introduce another invariant, which comes from the \emph{very discrete world of connected graphs thanks to a procedure of continuization of discreteness}. The main application of this invariant is to lead to the discovery of the property SN, which is the first known purely metric property that reduces to amenability for finitely generated groups.

\subsection{Definition of the isoperimetric constant}\label{suse:defiso}

Let $X=(V,E)$ be an infinite, locally finite, connected graph. There are many different definitions of isoperimetric constant. We will start from the one that seems to us easier to be generalized. We will define locals isoperimetric constants $\iota_k$ and a global isoperimetric constant $\iota$. We will observe that the definition of $\iota$ would be the same, starting from one of the other definitions of isoperimetric constant accessible in literature.
\\\\
The following definition is probably due to McMullen (see \cite{McMu89}) and it also appears in \cite{Be-Sc97}, in \cite{Ce-Gr-Ha99} and \cite{El-So05}.

\begin{definition}\label{def:isoperimetricgraph}
Let $X=(V,E)$ be an infinite, locally finite, connected graph and let $V$ equipped with the \emph{shortest path} metric. The isoperimetric constant of $X$ is
\begin{align}\label{eq:isoperimetric}
\iota_1(X)=\inf\left\{\frac{|dB_1(A)|}{|A|}, A\subseteq V \text{ is finite and non-empty}\right\}
\end{align}
\end{definition}

Generalizing Definition \ref{def:isoperimetricgraph}, Elek and S\`{o}s proposed in \cite{El-So05} the following

\begin{definition}\label{def:kisograph}
Let $X=(V,E)$ be an infinite, locally finite, connected graph and let $V$ equipped with the \emph{shortest path} metric. For any integer $k\geq1$, define the $k$-isoperimetric constant of $X$ to be
\begin{align}\label{eq:kisoperimetric}
\iota_k(X)=\inf\left\{\frac{|dB_k(A)|}{|A|}, A\subseteq V \text{ is finite and non-empty}\right\}
\end{align}
\end{definition}

An observation is needed at this point. It is clear that these definitions were different, since they did not know the discrete boundary. Indeed they used $cN_k(A)\setminus A$ instead of $dN_k(A)$, which is the same, when the graph is connected and the metric is the shortest path metric (for a proof, use Lemma \ref{lm:nojump2}, which holds also for locally finite connected graphs). Seen in this way, these definitions do not use the graph structure, but only the metric structure. So, theoretically, we could have used them as definitions even in our case. The point is that they would not have been an NPP-invariant. Since NPP isomorphisms of locally finite metric space preserves the discrete boundary, the trick is indeed to use it.

\begin{definition}\label{def:isoperimetricspace}
Let $(X,d)$ be a locally finite metric space. The \textbf{$k$-isoperimetric constant of $X$} is
\begin{align}\label{eq:isoperimetric2}
\iota_k(X)=\inf\left\{\frac{|dB_k(A)|}{|A|}, A\subseteq X \text{ finite and non-empty}\right\}
\end{align}
\end{definition}

Notice that the notion of $k$-isoperimetric constant is by definition a local notion, since it looks at the discrete $k$-neighborhood of the sets. Sometimes, it is better to have a global notion.

\begin{definition}\label{def:globaliso}
Let $(X,d)$ be a locally finite metric space. The \textbf{isoperimetric constant of $X$} is
\begin{align}\label{eq:isoglobal}
\iota(X)=\sup_{k\geq0}\inf\left\{\frac{|dB_k(A)|}{|A|}, A\subseteq X \text{ finite and non-empty}\right\}
\end{align}
\end{definition}

The following result is now obvious \emph{by definition} (and by Proposition \ref{prop:nppdiscrete}), but we state it as a theorem because of the intrinsic importance to have an invariant.

\begin{theorem}\label{th:isoinvariant}
The $k$-isoperimetric constant $\iota_k$ is an NPP-invariant. Consequently, also the isoperimetric constant $\iota$ is an NPP-invariant.
\end{theorem}

\begin{remark}
As we told at the beginning of this section, there are other notions of isoperimetric constant. Using these other notions, probably we would have obtained other symbolic $k$-isoperimetric constant, but the same global symbolic isoperimetric constant. Indeed all these definitions differ from each other locally.
\begin{itemize}
\item The definition that we have used makes a comparison between $A$ and $d_1(A)$, namely, in the language of Graph Theory, the set of vertices outside $A$ which are connected by an edge to some element of $A$.
\item Another definition, as in \cite{Do84} and \cite{Co-Sa93}, makes use of the vertices inside $A$.
\item Ornstein-Weiss' definition makes use both of the vertices both inside and outside $A$ (see \cite{Or-We87})
\end{itemize}
All these differences vanish making use of the parameter $k$.
\end{remark}

Now, it is well-known by a result of Ceccherini-Silberstein, Grigorchuk and de la Harpe (for finitely generated group) and then by a result of Elek and S\'{o}s (for locally finite graphs), that the condition $\iota(X)=0$ is equivalent to the notion of amenability. The natural question is then: \emph{What metric property corresponds to have $\iota(X)=0$, where $X$ is a locally finite metric space?} This question leads to the first known purely metric property which reduces to amenability for finitely generated groups. We will discuss this property, called property SN, in the next section.

\subsection{Some observation about the Markov operator and the Laplacian}\label{suse:laplacian}

At this point one is tempted to define the Markov operator, the Laplacian, the Gradient and so on in such a way to have some isoperimetric inequalities.

Let us recall the definition of the Markov operator for a connected graph of bounded degree equipped with the shortest path distance. Let $\mathcal H$ be the Hilbert space of functions $h:X\rightarrow\mathbb C$ such that $\sum_{x\in X}deg(x)|h(x)|^2<\infty$.

\begin{definition}
The Markov operator is the self-adjoint and bounded operator $T$ on $\mathcal H$ defined by
\begin{equation}\label{eq:markov}
(Th)(x)=\frac{1}{deg(x)}\sum_{y\in B_1(x)}h(y)
\end{equation}
where $B_1(x)=\{y\in X:d(x,y)=1\}$.
\end{definition}

The point is that it is not clear how to replace $deg(x)$ and $B_1(x)$, in order to have some isoperimetric inequalities. Indeed there is some ambiguity due to the following fact: the isoperimetric constant looks, in some sense, outside the sets. Namely, given a finite set $F$, it looks at the behavior of $\frac{|dB_1(F)|}{|F|}$. But the points in $dB_1(F)$ are not the nearest points to some point in $F$, but they are the points that are nearest to $F$ when an observer sitting on $F$ looks outside $F$. Now, the ambiguity should be clear. In order to have a Markov operator which agrees with the isoperimetric constant, one should replace $B_1(x)$ with some set which takes into account also the nearest points \emph{outside}... outside what? Our idea is that in this case one should fix a base point $x_0\in X$, consider the increasing and covering sequence of sets $\{x_0\}\subseteq dB_1(\{x_0\})\subseteq dB_2(\{x_0\})\subseteq\ldots$ and define the Markov operator $T_{x_0}$, replacing $B_1(x)$ with the union between $dB_1(\{x\})$ and the set of nearest point to $x$ outside the first $dB_k(\{x_0\})$ containing $x$. This procedure seems to have a surprising philosophical affinity with Tessera's viewpoints technicality used to define analogous notions in any metric space (see \cite{Te08}). Unfortunately, we cannot adopt word by word Tessera's technique, that use classical balls, if we want to keep some relation with the isoperimetric constant, which use balls in the sense of discrete boundary. Indeed, as observed more or less explicitly many times, the discrete boundary of a set can be locally very different from the classical boundary defined using the balls. In many cases, it is well possible that the two boundaries differ very little at large scale. In this case it is expected that, \emph{at large scale}, \emph{our} Markov operator, \emph{our} Laplacian and \emph{our} gradient should be almost the same as the ones defined by Tessera. For instance, it is philosophically likely that the large scale Sobolev's inequalities (see \cite{Co96} and \cite{Te08}) are the same. This is of course an interesting field of research.

\section{The Small Neighborhood property}\label{se:sn}

In this section we construct a purely metric property (meaning that the definition can be given for any metric space) that corresponds, for locally finite metric spaces, to have isoperimetric constant equal to zero. This property will be called Small Neighborhood property (property SN). Indeed, the geometric intuition behind this property is that \emph{certain} big sets have \emph{negligible boundary}. This is the same intuition used to describe the amenability of a finitely generated group, through the celebrated F{\o}lner condition (see \cite{Fo55}). We will see, in fact, that there is a precise correspondence between property SN and amenability: the two are equivalent for Cayley graphs of finitely generated groups. Hence, as far as we know, the property SN turns out to be the first known purely metric description of amenability, solving, in part\footnote{\emph{In part} means that it is not clear what happens for non finitely generated groups.}, the longstanding question whether or not there is a purely metric description of amenability.

\subsection{Definition of Small Neighborhood property}\label{suse:defsn}

As already said, the idea behind the property SN is quite intuitive: we want that \emph{certain} subsets of a metric space have \emph{negligible boundary}. Trying to formalize this property, one faces some troubles. Indeed

\begin{itemize}
\item It is not clear which \emph{certain sets} we should consider. Of course not any. Indeed, whatever this property is, we want that the Euclidean metric on the real line has this property. So we need to exclude those sets, as the set of  rational numbers, that are \emph{intrinsically too big}.
\item We want a notion able to describe the amenability of a f.g. group; namely, we want that a f.g. group is amenable if and only if its Cayley graph has the property SN. Now the Cayley graph of a f.g. group is always a locally finite metric space and then the (topological) boundary of a bounded set is always empty, hence negligible in every sense. This means that we cannot adopt the usual definition of boundary, otherwise every (Cayley graph of a) f.g. group would have the property SN, against our aim.
\item A general metric space does not have a natural measure to define the notion of negligible set.
\end{itemize}

Fortunately, we are helped by the work done in the previous sections: we have a notion of boundary able to take in account, at the same time, discrete and continuous behaviors; on the other hand, we interpret \emph{negligibility} saying that the set is \emph{topologically} bigger than its $k$-boundary; namely, the $k$-boundary does not contain any copy of the set. Finally, as we said, we have to look just at those sets which are not intrinsically too big; namely, open and bounded sets whose complementary contains an open non-empty set. At this point we can make the idea completely formal very easily. Let us fix some notation
\begin{itemize}
\item given a subset $A$ of the metric space $X$ and a nonnegative integer $k$. Let us denote by $F(A,k)$ the set of continuous, open and injective maps $f:X\rightarrow X$ such that $f(A)\subseteq B_k(A)$.
\item let $\mathcal F$ denote the class of all subsets $A$ of the metric space $X$ which are open, bounded and such that $X\setminus A$ contains an open non-empty set.
\end{itemize}

\begin{definition}\label{def:small}
A metric space $(X,d)$ is said to have the \textbf{Small Neighborhood property} if for all nonnegative integers $k$ there exists $A\in\mathcal F$ such that $F(A,k)=\emptyset$.
\end{definition}

\subsection{Examples of spaces with or without the property SN}\label{suse:exsn}

In this subsection we want to collect some examples of spaces with or without the property SN. It is an occasion to introduce a new class of spaces, the pitless spaces.\\

Let $x\in X$ and $r>0$. Denote by $B(x,r)=\{y\in X:d(x,y)<r\}$ and $N(x,r)=\{y\in X:d(x,y)\leq r\}$. It is always $\overline{B(x,r)}\subseteq N(x,r)$, but it is well known that the inclusion could be dramatically proper. The property $N(x,r)=\overline{B(x,r)}$, for all $x\in X$ and $r>0$ is known to be equivalent to the property that for any $x\in X$, the mapping $y\in X\rightarrow d(x,y)$ has a unique local proper minimum at $x$. This means that we can move from a point $y$ to a point $x$ without \emph{going back}: in some sense, the space has no \emph{pits}.

\begin{definition}\label{def:tips}
A \textbf{pit} on a metric space $(X,d)$ is a point $y$ for which there exists another point $x\in X$ such that the mapping $z\in X\rightarrow d(z,x)$ has a local proper minimum at $y$. A metric space is called \textbf{pitless} if it has no pits.
\end{definition}

\begin{remark}
It is clear that pitlessness implies continuity (see Definition \ref{def:continuous}). Indeed, if $X$ is not continuous, let $A$ be an open and bounded set for which there is a complete chain $d(x_0,A)<d(x_1,A)$ inside $\{d(x,A),x\in X\}$. Consider $cN_{d(x_1,A)}(A)^\circ$, which coincides with $\overline A$ and then its closure cannot be equal to $cN_{d(x_1,A)}$. The converse is not true. To have an example of a continuous space (also with property SN) and pits it suffices to consider suitable subsets of $\mathbb R^2$; for instance, the space $X$ constructed as follows:
\begin{itemize}
\item let $R$ the rectangle defined by the points $(-1,0),(-1,2),(0,2),(0,0)$.
\item let $X$ be the intersection between $R$ and the epigraph of a regular function $f:[-1,0]\rightarrow\mathbb R$ such that $f(-1)=f(0)=0$, $f(x)>\sqrt{1-x^2}$, in a neighborhood of $-1$, $f(x)<\sqrt{1-x^2}$ in a neighborhood of $0$ and $f(x)=\sqrt{1-x^2}$ only in a single point.
\end{itemize}
It is clear that $X$ is a continuous space with property SN and tips.
\end{remark}

\begin{proposition}\label{prop:example}
\begin{enumerate}
\item Every pitless space has the property SN.
\item The metric space $(\mathbb R,d)$, with $d(x,y)=\min\{1,|x-y|\}$ does not have the property SN.
\end{enumerate}
\end{proposition}

\begin{proof}
\begin{enumerate}
\item As already said, pitlessness is equivalent to the property: $N(x,r)=\overline{B(x,r)}$, for all $x\in X$ and $r>0$. This latter property clearly implies that $cB_\alpha(A)\setminus\overline{cB_\alpha(A)^\circ}=\emptyset$, for all $\alpha$. Using continuity (see Definition \ref{def:continuous}) we get $B_k(A)=\overline A\setminus A$, that clearly cannot contain any open set.
\item Let $k=2$ and $A\in\mathcal F$. The space under consideration is clearly continuous, and so $dB_k(A)=\overline A\setminus A$. On the other hand, a direct computation, shows that
$$\bigcup_{0\leq\alpha\leq k-1}(cN_\alpha(A)\setminus\overline{cN_\alpha(A)^\circ})=\mathbb R\setminus\overline A$$
It follows that $B_k(A)=\mathbb R\setminus A$, which contains open sets. It is clear that $A$ can be mapped in such open sets, since the topology is the same as the Euclidean topology.
\end{enumerate}
\end{proof}

\begin{remark}
\begin{enumerate}
\item Observing that Banach spaces are pitless, we get that any Banach space has the property SN.
\item Observe that $(\mathbb R,d)$, with $d(x,y)=\min\{|x-y|,1\}$, is continuous. So pitlessness is in this case a crucial properties, much stronger than continuity.
\item Proposition \ref{prop:example} shows also that the property SN is not invariant under quasi-isometries. Indeed the metric space $(\mathbb R,d)$ with $d(x,y)=\min\{1,|x-y|\}$ does not have the property SN, but it is bounded and then quasi-isometric to a singleton, which has the property SN.
\end{enumerate}
\end{remark}
\subsection{Correspondence between amenability and property SN}\label{suse:amenabilityvssn}

As we told, the idea behind the property SN is philosophically the same as the idea behind the amenability of a finitely generated group. So the purpose of this section is to clarify this relation. In particular, we prove that a finitely generated group is amenable if and only if its Cayley graph, equipped with its natural metric, has the property SN. \\
For the sake of completeness, let us recall the definition of the Cayley graph of a group. Given a finitely generated group $G$, with a generating set $S$ that we can suppose to be symmetric and not containing the identity, its Cayley graph $\Gamma$ is defined as follows:

\begin{itemize}
\item every element $g\in G$ is a vertex of $\Gamma$,
\item two vertex $g$ and $h$ are joined by an edge if and only if there is $s\in S$ such that $h=gs$.
\end{itemize}

The natural \emph{shortest path} distance on $\Gamma$ is called word metric on $G$ with respect to the generating set $S$ and it is denoted by $d_S$. Since $S$ is finite, the metric space $(G,d_S)$ is locally finite and this is fundamental to introduce the doubling condition in Definition \ref{def:amenable}. In the sequel, $S$ will be always a fixed symmetric generating subset of $G$ non-containing the identity.

Finally we recall the notion of amenability of a f.g. group through an equivalent condition proposed by Ceccherini-Silberstein, Grigorchuk and de la Harpe (see \cite{Ce-Gr-Ha99}, Theorem 32).

\begin{definition}\label{def:amenable}
A f.g. group $G$ is amenable if and only if the metric space $(G,d_S)$ does not verify the doubling condition\footnote{A locally finite metric space satisfies the \textbf{doubling condition} if there exists a constant $k>0$ such that
\begin{equation}\label{eq:doubling}
|cN_k(A)|\geq2|A|
\end{equation}
for any non-empty finite subset $A$ of $G$.}.
\end{definition}

\begin{lemma}\label{lm:nojump}
Let $A$ be a finite non-empty subset of $G$, $g\in G\setminus A$ and denote $d(g,A)=l$. For any $l'\in\{0,\ldots,l\}$, there is $g'\in G$ such that $d(g',A)=l'$
\end{lemma}

\begin{proof}
Since $A$ is finite, there is $a\in A$ and $s_1,\ldots,s_l\in S$ such that $a=s_1\cdot\ldots\cdot s_l$ and we can assume that this is already a reduced word. It suffices to take $g'=as_1\cdot\ldots\cdot s_{l'}$.
\end{proof}

\begin{lemma}\label{lm:nojump2}
Let $A$ be a finite subset of $G$. For all $k\geq1$, one has
$$
dB_k(A)=cB_{k-1}(A)
$$
\end{lemma}

\begin{proof}
Let $x\in cN_{k-1}(A)$ and $a\in A$ realizing the minimum of $d(x,A)$. Let us say $d(x,a)=l\leq k-1$. Using lemma \ref{lm:nojump} we can find a sequence $a=x_0,x_1,\ldots,x_l=x$ such that $d(A,x_i)=i$. Since the metric $d$ attains only integer values, it follows that the chain $d(x_0,A)<\ldots<d(x_l,A)$ is complete and it has length equal to $l+1\leq k$. Hence $x=x_l\in dB_k(A)$. Conversely, let $x\in dB_k(A)$. Therefore, there are $x_0,x_1,\ldots,x_l=x$, with $l\leq k-1$, such that the chain $d(x_0,A)<d(x_1,A)<\ldots<d(x_l,A)$ is complete. Completeness and Lemma \ref{lm:nojump} imply that $d(x_i,A)=i$ for all $i$ and in particular $d(x_l,A)=l\leq k-1$. It follows that $x=x_l\in cB_{k-1}(A)$.
\end{proof}

\begin{theorem}\label{th:snvsamenability}
A f.g. group $G$ is amenable if and only if the metric space $(G,d_S)$ has the property SN.
\end{theorem}

\begin{proof}
The proof is based on Lemma \ref{lm:nojump2}. Indeed, first of all it implies that $\alpha(k)=k$ and then
\begin{align}
\bigcup_{0\leq\alpha\leq k-1}\left(cN_\alpha(A)\setminus\overline{cN_\alpha(A)^\circ}\right)=c
B_k(A)
\end{align}
It follows, applying Lemma \ref{lm:nojump2}, that
\begin{align}
B_k(A)=dB_k(A)\cup\bigcup_{0\leq\alpha\leq k-1}\left(cN_\alpha(A)\setminus\overline{cN_\alpha(A)^\circ}\right)=cB_k(A)
\end{align}
Now we can prove the theorem. Let us suppose that $G$ is amenable. Fix $k\geq1$ and let $A$ be a finite subset of $G$ such that $|cN_k(A)|<2|A|$. So there is no injective mapping from $A$ to $cN_k(A)$ and, from the observation above, $F(A,k)=\emptyset$. Conversely, assume that $(G,d_A)$ has the property SN. Let $k\geq0$ and assume that for all $A$, one has $|N_k(A)|\geq2|A|$. Therefore there is an injective mapping $f:A\rightarrow N_k(A)\setminus A$, which is automatically continuous and open, since the topology is discrete. By the initial observation we get $f\in F(A,k)$, contradicting the property SN.
\end{proof}

\begin{remark}
A longstanding question asks whether there is a purely metric condition for the amenability of a locally compact group, not necessarily finitely generated. A part authors' opinion that amenability, in general, is a measure-theoretical notion that becomes metric in some particular cases - namely, in the locally finite case - let us observe explicitly that the property SN is not able to capture amenability in general. Indeed, the locally compact amenable group $(\mathbb R,d)$, with $d(x,y)=\min\{|x-y|,1\}$, does not have the property SN. Looking at Proposition \ref{prop:example}, this happens because we have forced the space to have pits. This is why, it would be nice to have a better understanding of the following question
\begin{problem}
Given a locally compact amenable group which is metrizable by a continuous metric. Is it always possible to find an equivalent pitless metric? Conversely, it is not clear whether or not there exists a non-amenable metrizable group with the property SN.
\end{problem}
\end{remark}

\subsection{Correspondence between property SN and the other notions of amenability}\label{suse:amenabilityvssn2}

In the previous section we have proved that the property SN is a purely metric property (meaning that it can be formulated for any metric space) which generalizes amenability of a finitely generated group. The idea that such a purely metric condition has to exist is not new and it relies, in fact, in F{\o}lner's characterization of amenability \cite{Fo55}, which can be intuitively explained saying that a finitely generated group is amenable if and only if, looking at its Cayley graph, the boundary of arbitrarily big set is small. There have been, indeed, many tentative to find this purely metric condition, but all of them unsatisfactory, since applicable just to locally finite metric spaces (see \cite{Ce-Gr-Ha99}) or to metric spaces with \emph{coarse bounded geometry} (see \cite{Bl-We92}). From another point of view, Laczkovich generalized the notion of paradoxical decomposition to any metric space (\cite{La90}, \cite{La01} and also \cite{De-Si-So95}). It turns out that $\mathbb R$ is paradoxical and this makes sense thinking of translation invariant $\sigma$-additive measures defined on the power set of $\mathbb R$, but it does not take into account the existence of translation invariant finitely additive measures defined on a \emph{nice} algebra of subsets. The latter is indeed how amenability behaves in the \emph{continuous} case.

In this section we want to make a comparison among the property SN and the other two notions of amenability that have been proposed in the past for metric spaces that are more general than the ones arising as the Cayley graph of a group.
\\\\
The first definition is taken from \cite{Ce-Gr-Ha99}. In their paper, it is presented in a different way, but the reader can easily reconstruct the following definition using their Definition 5 and Theorem 32. In particular, their Theorem 32 gives many other equivalent definitions.

\begin{definition}\label{def:amenable2}
A locally finite metric space is said to be amenable if the doubling condition is not satisfied.
\end{definition}

The second notion of amenability, proposed by Block and Weinberger (see \cite{Bl-We92}) is a bit more technical.

\begin{definition}
A metric space $(X,d)$ is said to have \textbf{coarse bounded geometry} if it admits a quasi-lattice; namely, there is $\Gamma\subseteq X$ such that
\begin{enumerate}
\item There exists $\alpha>0$ such that $cN_\alpha(\Gamma)=X$,
\item For all $r>0$ there exists $K_r>0$ such that, for all $x\in X$,
\begin{align}
\left|\Gamma\cap B_r(x)\right|\leq K_r
\end{align}
where $B_r(x)$ is the open ball with radius $r$ about $x$.
\end{enumerate}
\end{definition}

\begin{definition}\label{def:amenable3}
A metric space $(X,d)$ with coarse bounded geometry, given by the quasi-lattice $\Gamma$, is said to be amenable if for all $r,\delta>0$, there exists a finite subset $U$ of $\Gamma$ such that
\begin{equation}\label{eq:folner}
\frac{|\partial_rU|}{|U|}<\delta
\end{equation}
where $\partial_rU=\{x\in\Gamma:d(x,U)<r, d(x,\Gamma\setminus U)<r\}$.
\end{definition}

The natural question is about the possible logic implications among these two notions of amenability and the property SN. The following discussion gives an almost complete summary of the situation.\\

Let us start from comparing Definitions \ref{def:amenable2} and \ref{def:amenable3}. The mysterious Remark 42 in \cite{Ce-Gr-Ha99} contains some ideas that we want to make formal. First of all, the fact that there are locally finite metric spaces that do not have coarse bounded geometry\footnote{For instance, attach over any natural number $n$ a set $A_n$ with $n$ elements and define the metric to be 1 between points in the same $A_n$ and $n+m$ between a point in $A_n$ and one in $A_m$, with $n\neq m$.} and, conversely, there are spaces with bounded geometry that are not locally finite, forces to consider the problem in the restricted class of locally finite spaces with coarse bounded geometry. We have the following result:

\begin{proposition}\label{prop:ceccherinivsblock}
Let $(X,d)$ be a locally finite metric space with coarse bounded geometry. It is amenable in the sense of Definition \ref{def:amenable2} if and only if it is amenable in the sense of Definition \ref{def:amenable3}.
\end{proposition}

\begin{proof}
Let $\Gamma\subseteq X$ be a quasi-lattice. By definition of quasi-lattice, $X$ is quasi-isometric to $\Gamma$. Since amenability in the sense of Definition \ref{def:amenable2} is invariant under quasi-isometries (see \cite{Ce-Gr-Ha99}, Proposition 38), we can restrict our attention on $\Gamma$. The result follows by the observation that the F{\o}lner condition on $\Gamma$ in Equation \ref{eq:folner} is equivalent to the negation of the doubling condition on $\Gamma$ in Equation \ref{eq:doubling}.
\end{proof}

About the relation between Property SN and amenability in the sense of Definition \ref{def:amenable2} we can give the following

\begin{proposition}\label{prop:caprarovsceccherini}
If $(X,d)$ comes from a locally finite connected graph, then it has the property SN if and only if it is amenability in the sense of Definition \ref{def:amenable2}.
\end{proposition}

\begin{proof}
It suffices to observe that Lemma \ref{lm:nojump2} holds in this case and that the proof of Theorem \ref{th:snvsamenability} does not really depend on the fact that $X$ is the Cayley graph of a group.
\end{proof}

\begin{remark}\label{rem:caprarovsceccherini}
In general, it is clear that there are locally finite metric spaces with property SN which are non amenable in the sense of Definition \ref{def:amenable2}. A simple instance is given by the space $X$ constructed as follows: let $\mathbb F_2=<a,b>$ be the free group on two generators and consider the usual Cayley graph with respect to the generating set $\{a,a^{-1},b,b^{-1}\}$. Now get rid of all balls of radius which is not a power of $2$. The space obtained is amenable in the sense of Definition \ref{def:amenable2}, but it does not have the property SN.\\

On the other hand, it is not clear at the moment if the other implication can be true: are there non amenable locally finite space with the property SN?
\end{remark}

About the relation between property SN and amenability in the sense of Definition \ref{def:amenable3} we can say right now that the two are, in general, distinct properties. Indeed the space $(\mathbb R,d)$ with $d(x,y)=\min\{1,|x-y|\}$ does not have the property SN, but it is amenable in the sense of Block and Weinberger, taking $\Gamma=\{0\}$. On the other hand, if we consider the additional hypothesis that the space is locally finite, then Definitions \ref{def:amenable2} and \ref{def:amenable3} are basically the same, by Proposition \ref{prop:ceccherinivsblock}, and so Proposition \ref{prop:caprarovsceccherini} and Remark \ref{rem:caprarovsceccherini} apply.

\subsection{Description of property SN using the symbolic isoperimetric constant}\label{suse:snvsiso}

In this final subsection, we give the promised characterization of property SN in terms of the isoperimetric constant. This characterization is the equivalent of results already proved for connected locally finite graphs by Ceccherini-Silberstein, Grigorchuk and de la Harpe (see \cite{Ce-Gr-Ha99}, Theorem 51) and by Elek and S\'{o}s (see \cite{El-So05}, Proposition 4.3).

\begin{theorem}\label{th:isoperimetricvssn}
A locally finite metric space $(X,d)$ has the property SN if and only if $\iota(X)=0$.
\end{theorem}

\begin{proof}
$\iota(X)>0$ if and only if there is $k$ such that
\begin{align}
\inf\left\{\frac{|dB_k(A)|}{|A|}, A\subseteq X \text{finite and non-empty}\right\}>0
\end{align}
It is now a standard argument to prove that there is (another) $k$ such that $|dB_k(A)|\geq|A|$ for all $A\subseteq X$ finite and non-empty (see for instance \cite{Ce-Gr-Ha99}, just below Definition 30). Now, since $X$ is locally finite, $B_k(A)$ has no continuous part, namely $B_k(A)=dB_k(A)$ (see Lemma \ref{lem:discrete}). It follows that $\iota(X)>0$ if and only if there is a positive integer $k$ such that $|B_k(A)|\geq|A|$ for all $A\subseteq X$ and this is the negation of the property SN, as in the proof of Theorem \ref{th:snvsamenability}.
\end{proof}

From this result we can deduce a first application of the NPP-invariance of the isoperimetric constant.

\begin{corollary}\label{cor:sninvariant}
For locally finite metric spaces, the property SN is invariant under NPP-isomorphisms.
\end{corollary}

Observe that amenability in the sense of Definitions \ref{def:amenable2} and \ref{def:amenable3} is not invariant under NPP-isomorphisms. This does not mean that our property SN is better, as a generalization of amenability, than the others, but just that those generalizations do not catch the essence of the problems that we are considering and then we need something else.

\section{The zoom isoperimetric constants}\label{se:zoomiso}

In this section we want to introduce another NPP-invariant, which is, in some sense, more precise than the isoperimetric constant. In fact, the isoperimetric constant looks at the best behavior inside the metric space and so the information encoded in it could be very far from what a local observer really sees. In order to make this obscure sentence clearer, let us consider the following example. Consider a regular tree of degree $d\geq3$ and attach over a vertex a copy of the integers. We get a connected graph of bounded degree with isoperimetric constant equal to zero. This graph is then amenable in every sense: it has property SN and it is amenable in the sense of Definitions \ref{def:amenable2} and \ref{def:amenable3}. This happens because the isoperimetric constant forgets what happens in the non-amenable part. On the other hand, consider an observer living in the vertex $x$. His/her way to observe the universe where (s)he lives is to \emph{increase his/her knowledge step by step}, considering at the first step the set $A_0=\{x\}$, then $A_1=dN_1(\{x\})$ and so on: his/her point of view on the universe is given by the sequence $A_n=dN_n(\{x\})$. It is clear that, wherever $x$ is, the sequence $A_n$ is eventually doubling, in the sense that $|A_n|\geq2|A_{n-1}|$. So, the isoperimetric constant says that the universe is amenable, but every local observer would say that the universe is non-amenable. This lack of agreement is the reason why in this section we try to introduce a new invariant which takes into account all these local observations. By the way, this invariant turns out to be new even for locally finite graphs.\\

\subsection{Definition of zoom isoperimetric constants}\label{suse:defzoom}

Let us fix some notation. Let $(X,d)$ be a locally finite metric space, $x\in X$ and $k,n$ be two positive integers. We define

\begin{align}
\zeta_k(x)=\inf\left\{\frac{|dN_{nk}(x)|}{|dN_{(n-1)k}(x)|}, n\geq1\right\}
\end{align}

Then define

\begin{align}
\zeta(x)=\sup\left\{\zeta_k(x),k\geq1\right\}
\end{align}

\begin{definition}\label{def:zoom}
The \textbf{upper zoom isoperimetric constant} of $X$ is
\begin{align}\label{eq:uzoomconstant}
\zeta^+(X)=\sup\left\{\zeta(x), x\in X\right\}
\end{align}
The \textbf{lower zoom isoperimetric constant} of $X$ is
\begin{align}\label{eq:lzoomconstant}
\zeta^-(X)=\inf\left\{\zeta(x), x\in X\right\}
\end{align}
\end{definition}

Also in this case, the following result is obvious by construction and Proposition \ref{prop:nppdiscrete}.

\begin{theorem}\label{th:zoominvariant}
The zoom isoperimetric constants are invariant under NPP-isomorphisms.
\end{theorem}

Let us introduce the parameter $\lambda(X)=\zeta^+(X)-\zeta^-(X)$. In our opinion, the case $\lambda(X)=0$ is particularly interesting, since it describes the situation when the observations made by two arbitrary observers agree. Unfortunately, it can happen that $\lambda(X)\neq0$, as shown in the next example (which is a variation of an example suggested by Ilya Bogdanov). In the next subsection we will show a seemingly interesting example when $\lambda(X)=0$.

\begin{example}
Consider a binary tree with a root vertex $a$, that is, $a$ has degree $1$ and each other vertex has degree $3$. For some integer of the shape $n_1!$ large enough, delete some vertices (and the subtree hanging from them) from the $n_1!$-th layer in such a way that
$$
\frac{|dB_{n_1!-1}(\{a\})|}{|dN_{n_1!-2}(\{a\})|}\approx\frac{1}{1000}
$$
and observe that
$$
\frac{|dB_{n_1!}(\{a\})|}{|dN_{n_1!-1}(\{a\})|}\geq\frac{1}{800}
$$
Now, for an integer of the shape $n_2!$ much larger than $n_1!$, delete vertices from the $n_2!$-layer in the same manner, and so on. For any $k\in\mathbb N$, we get
\begin{equation}\label{eq:ja}
\frac{|dB_{n_k!-1}(\{a\})|}{|dN_{n_k!-2}(\{a\})|}\approx\frac{1}{1000}
\end{equation}
and
\begin{equation}\label{eq:jb}
\frac{|dB_{n_k!}(\{a\})|}{|dN_{n_k!-1}(\{a\})|}\geq\frac{1}{800}
\end{equation}
Finally, glue two copies of this graph at the vertex $a$. Now, from the formula in \ref{eq:ja}, we get $\zeta_k(a)\leq1+\frac{1}{1000}$, for all $k$ (since the sequence of the multiple of $k$ eventually intersects the sequence $n_k!$). On the other hand, if $b$ is one of the neighbors of $a$, the formula in \ref{eq:jb} gives $\zeta_k(b)>1+\frac{1}{1000}$.
\end{example}

\subsection{Locally amenable metric spaces}\label{suse:completelyamenable}

It is clear that $0\leq\iota(X)\leq\zeta^-(X)-1\leq\zeta^+(X)-1$ and then the condition $\zeta^+(X)=1$, implies the property SN. On the other hand, the condition $\zeta^+(X)=1$, means $\zeta(x)=1$, for all $x\in X$ and this is quite a strong condition. It means that the space is seen to be amenable by any local observer.

\begin{definition}\label{def::completelyamenable}
A locally finite metric space is called \textbf{locally amenable} if $\zeta^+(X)=1$.
\end{definition}

The following result is basic, but it has a nice interpretation: in some cases, it suffices that an observer sees the space amenable, to conclude that the space is amenable for any observer.

\begin{theorem}\label{th:localglobalamenability}
Let $(X,d)$ be a graph-type locally finite metric space. Suppose there exists $x\in X$ such that $\zeta(x)=1$, then $\zeta^+(X)=1$ (and, obviously, $\lambda(X)=0$)
\end{theorem}

\begin{proof}
Let $y\in X$, we have to prove that also $\zeta(y)=1$. Let $k$ be fixed and let $n_0$ such that $x=A_0^h(\{x\})\in A_{n_0-1}^k(y)$ (notice that $h$ is still undetermined, since $A_0^h(\{x\})=x$ holds for every $h$). Let $h$ such that $A_1^h(\{x\})\supseteq A_{n_0}^k(y)$. Now, since $X$ is of graph-type, then the discrete neighborhoods behave like the balls. It follows that for all $m\in\mathbb N$, one has
$$
\left\{
  \begin{array}{ll}
    A_{m-1}^h(\{x\})\subseteq A_{n_0+m-1}^k(y)\\
    A_m^h(\{x\})\supseteq A_{n_0+m}^k(y)
  \end{array}
\right.
$$
which clearly implies that
\begin{align}
\frac{|A_{n_0+m}^k(y)|}{|A_{n_0+m-1}^k(y)|}\leq\frac{|A_m^h(x)|}{|A_{m-1}^h(x)|}
\end{align}

Now, the left hand side is $\geq1$ and, by hypothesis, the right hand side has infimum equal to $1$, independently on $h$. It follows that also the left hand side has infimum equal to $1$, as required.
\end{proof}

Now we want give a first proposition about the relation between the growth rate of a finitely generated group and the local amenability of its Cayley graph. Let $G$ be a finitely generated group, fix a finite symmetric generating set $S$ and consider the locally finite metric space $(G,d_S)$. Denote by $B(n)$ the ball of radius $n$ about the identity. We recall the following definition:

\begin{itemize}
\item A group is said to have \textbf{polynomial growth} if there is a positive integer $k_0$ such that $|B(n)|\leq n^{k_0}$, for all $n$.
\item A group is said to have \textbf{exponential growth} if there is $a>1$, such that $|B(n)|\geq a^n$, for all $n$. A group which does not have exponential growth is said to have sub-exponential growth.
\end{itemize}

It is well-known that there is no perfect correspondence between amenability and growth rate. Indeed, groups of sub-exponential growth are amenable, but there are amenable groups generated by $m$ element whose growth approach the one of the free group on $m$ generators (see \cite{Ar-Gu-Gu05}). Here we want to try to obtain a correspondence between growth rate and local amenability.

\begin{proposition}\label{prop:locallyamenable}
\begin{enumerate}
\item If $G$ has polynomial growth, then $(G,d_S)$ is locally amenable.
\item If $(G,d_S)$ is locally amenable, then $G$ has sub-exponential growth.
\end{enumerate}
\begin{proof}
\item If $G$ has polynomial growth, then there exists a positive integer $k_0$ such that $B(n)\sim n^{k_0}$ (this basically follows by Gromov's theorem\footnote{Gromov's theorem\cite{Gr81} states that a group of polynomial growth contains a nilpotent group of finite index and then, a posteriori, the growth rate of the original group is not just bounded above by a polynomial, but it is exactly polynomial.}). Then, for all positive integer $k$, one has
$$
\frac{|B((n+1)k)|}{|B(nk)|}\sim\frac{((n+1)k)^{k_0}}{(nk)^{k_0}}\rightarrow1
$$
hence $\zeta(e)=1$. By Theorem \ref{th:zoominvariant}, it follows that $\zeta(G)=1$, hence $G$ is locally amenable.
\item Suppose that $(G,d_S)$ is locally amenable and suppose, by contradiction, that it has exponential growth. Then the limit of the sequence $\sqrt[n]{|B(n)|}$, which always exists (see, for instance, \cite{Ha00}, VI.C, Proposition 56), is equal to some $a>1$ (this is just another way to define groups with sub-exponential growth. See, for instance, the introduction of \cite{Ar-Gu-Gu05}). It follows that $|B(n)|\sim a^n$ and then
\begin{align}
\inf\left\{\frac{|B(n+1)|}{|B(n)|,n\in\mathbb N}\right\}>1
\end{align}
It follows that $\zeta(e)>1$ and, by Theorem \ref{th:zoominvariant}, getting the contradiction that $(G,d_S)$ is not locally amenable.
\end{proof}
\end{proposition}

\begin{problem}\label{prob:intermediategrowth}
Does there exist a finitely generated group with intermediate growth such that $(G,d_S)$ is locally amenable?
\end{problem}




\subsection{Expanders and geometric property (T)}\label{suse:expanders}

In this subsection we want to describe another connection between our theorization and one of the recent fields of research in Pure and Applied Mathematics, which is the theory of expanders and the geometric property (T).

In the previous section, we have proved that for graph-type metric spaces, the local condition $\zeta(x)=1$, for a particular point $x$, is actually a global condition, holding for any point. Conversely, also the local condition $\zeta(x)>1$ is actually global. We can see this in another way: start from the point $x$ and construct a sequence of graph $\mathcal G_0=\{x\}$, $\mathcal G_1=dN_1(\{x\})$ and so on. If $\zeta(x)>1$, then, apart the hypothesis on the degree, we obtain an expander (see \cite{Lu11} for an introduction to expanders from both pure and applied point of view). Since $\zeta(y)>1$ for all $y$, it means that we can say that the space is expander, without ambiguity. Analogously, we think that we can also generalize Willett-Yu's geometric property (T) to any graph-type metric spaces (for the definition of geometric property (T), see \cite{Wi-Yu10b}. For basics and motivation, see \cite{Wi-Yu10a} and \cite{Oy-Yu09}). In order to do that, we need a good definition of the Laplacian. Find such a definition is probably the most important basic problem that we left open and certainly is one of the purpose of the second article on this subject.

\subsection{Independence between $\iota(X)$ and $\zeta^-(X)$}\label{suse:independence}

Until now we have introduced two isoperimetric constants for locally finite metric spaces. The relation $\iota(X)\leq\zeta^-(X)-1$, says that the two invariants are not completely independent, meaning that when $\zeta^-(X)=1$, we can deduce the value of $\iota(X)$. In this section, we want to observe explicitly that this is basically the unique case: in general, there is no way to deduce a bound for the value of $\zeta^-(X)$, knowing $\iota(X)$.

It seems that there is no notion of independence between invariants in literature. So let us propose a definition that looks quite natural and easy to verify.

\begin{definition}
Let $I_1,I_2$ be two real-valued invariants (of some category). We say that $I_2$ is \textbf{upper independent} on $I_1$ if there is no real valued function $f_2$ such that
\begin{equation}
I_2(X)\leq f_2(I_1(X))
\end{equation}
for all objects $X$ in the category.

Analogously one can define lower independence and completely independence.
\end{definition}

\begin{theorem}\label{th:invariantindependent}
$\zeta^-$ is upper independent on $\iota$.
\end{theorem}

\begin{proof}
Consider a regular tree of degree $n$ and attach a copy of $\mathbb Z$ over a vertex. We obtain an infinite, locally finite, connected graph. Let $(X,d)$ be this graph equipped with the shortest path metric. In this case $\iota(X)=0$, independently on $n$, since it is enough to take finite sets concentrated in the copy of $\mathbb Z$. On the other hand it is clear that $\zeta^-(X)$ can be made arbitrarily big with $n$.
\end{proof}

\end{document}